\documentclass[english,11pt]{article}
\usepackage[T1]{fontenc}
\usepackage{babel}
\usepackage{amsmath,amsfonts,amsthm}
\usepackage{color}

\DeclareMathSymbol{\leqslant}{\mathalpha}{AMSa}{"36} 
\DeclareMathSymbol{\geqslant}{\mathalpha}{AMSa}{"3E} 
\renewcommand{\leq}{\;\leqslant\;}                   
\renewcommand{\geq}{\;\geqslant\;}                   

\newtheorem{Th}{Theorem}
\newtheorem{Le}[Th]{Lemma}
\newtheorem{Pro}[Th]{Proposition}

\newtheorem{Rq}[Th]{Remark}

\newcommand{\cA}{\ensuremath{\mathcal A}}
\newcommand{\cB}{\ensuremath{\mathcal B}}

\newcommand{\cD}{\ensuremath{\mathcal D}}

\newcommand{\cP}{\ensuremath{\mathcal P}}

\newcommand{\cR}{\ensuremath{\mathcal R}}

\newcommand{\cX}{\ensuremath{\mathcal X}}

\newcommand{\bbC}{{\ensuremath{\mathbb C}} }
\newcommand{\bbD}{{\ensuremath{\mathbb D}} }
\newcommand{\bbE}{{\ensuremath{\mathbb E}} }

\newcommand{\bbN}{{\ensuremath{\mathbb N}} }

\newcommand{\bbP}{{\ensuremath{\mathbb P}} }

\newcommand{\bbR}{{\ensuremath{\mathbb R}} }

\newcommand{\Om}{\Omega}
\newcommand{\E}{\bbE}

\newcommand{\N}{\bbN}
\newcommand{\Ne}{\bbN^{\ast}}
\newcommand{\R}{\bbR}
\newcommand{\bbd}{\mathbf{d}}

\newcommand{\tN}{\tilde{N}}
\newcommand{\hOm}{\hat{\Om}}
\newcommand{\hcA}{\hat{\cA}}
\newcommand{\hbbP}{\hat{\bbP}}
\newcommand{\hE}{\hat{\E}}

\newcommand{\crea}{\varepsilon^+}
\newcommand{\anni}{\varepsilon^-}

\title{The Lent Particle Method,\\ Application to Multiple Poisson Integrals}
\date{}
\author{Nicolas BOULEAU\\
Ecole des Ponts ParisTech\\
6 Avenue Blaise Pascal,
Marne-la-Vall\'ee 77455, France\\
{\tt bouleau@enpc.fr}}

\begin{document}

\maketitle


\begin{abstract}
We give a extensive account of a recent new way of applying  the Dirichlet form theory  to random Poisson measures. The main application is to obtain existence of density for the laws of random functionals of L\'evy processes or solutions of stochastic differential equations with jumps.  As in the Wiener case the Dirichlet form approach weakens significantly the regularity assumptions. The main novelty is an explicit formula for the gradient or for the ``carr\'e du champ" on the Poisson space called the lent particle formula because based on adding a new particle to the system, computing the derivative of the functional with respect to this new argument and taking back this particle before applying the Poisson measure. 

The article is expository in its first part and based on Bouleau-Denis \cite{bouleau-denis} with several new examples, applications to multiple Poisson integrals are gathered in the last part which concerns  the relation with the Fock space and some aspects of the second quantization.
\end{abstract}



\section{Introduction and framework.}
This lecture is an introduction to Dirichlet forms methods for studying regularity of random variables yielded by  L\'evy processes, solutions of stochastic differential equations driven by Poisson measures and multiple Poisson integrals. The main part of this study has been done in collaboration with Laurent Denis.

A Dirichlet forms is a generalisation of the classical quadratic operator
$\int_\Omega |\nabla f(x)|^2dx$
early introduced in potential theory. 
The concept has been developped especially by Beurling and Deny in the 1950's as an application of Hilbert space methods in potential theory, and by Fukushima in the 1970's in connection with symmetric Markov processes theory. It received recently a strong development in  infinite dimensional spaces where it appears as an alternative approach to Malliavin calculus.

The importance of the notion comes from the fact that if $P_t$ is a symmetric strongly continuous contraction semigroup on a space $L^2(\mu)$ (for $\mu$ $\sigma$-finite positive measure) with generator $A$, a necessary and sufficient condition that $P_t$ be Markov is that ``contractions operate" on the quadratic form
$\mathcal{E}[f]=-<Af,f>_{L^2(\mu)}$
i.e. $\mathcal{E}[\varphi(f)]\leq \mathcal{E}[f]$  for  $\varphi$  contraction from $\mathbb{R}$ to $\mathbb{R}$ (cf \cite{bouleau-hirsch2} Chap.I prop. 3.2.1). Such a quadratic form is called a Dirichlet form.

The case of Malliavin calculus is that of Wiener space taking for $P_t$ the Ornstein-Uhlenbeck semi-group. The corresponding Dirichlet form $\mathcal{E}$ possesses a ``carr\'e du champ" operator, i.e. may be written $\mathcal{E}[f]=\frac{1}{2}\int\Gamma[f]d\mu$  where $\Gamma$ is a quadratic operator from the domain of $\mathcal{E}$ to $L^1(\mu)$. This fact makes it possible the definition of a ``gradient" satisfying the chain rule and allowing a differential calculus through stochastic expressions and stochastic differential equations (SDE) and providing integration by parts formulae which yield existence of density results (cf \cite{malliavin}).

Using Dirichlet forms in this framework of Wiener space  improves several results : contraction arguments show that the Picard  iteration method for solving SDE's holds not only in $L^2$ but still for the stronger Dirichlet norm. This gives existence of density for solutions of SDE's under only Lipschitz assumptions on the coefficients (cf \cite{bouleau-hirsch1} and \cite{bouleau-hirsch2}). More generally, Dirichlet forms are easy to construct in the infinite dimensional frameworks encountered in probability theory and this yields a theory of errors propagation through the stochastic calculus (cf Bouleau \cite{bouleau3}),  also for numerical analysis of PDE and SPDE (cf Scotti \cite{scotti}).

As the Malliavin calculus has been extended to the case of  Poisson measures and SDE's with jumps, 
 either dealing with local operators acting on the size of the jumps (Bichteler-Gravereaux-Jacod \cite{bichteler-gravereaux-jacod}  Ma-R\"ockner\cite{ma-rockner2} L\'eandre \cite{leandre1} \cite{leandre2} etc.) or based on the Fock space representation of the Poisson space and finite difference operators (Nualart-Vives \cite{nualart-vives} Picard \cite{picard} Ishikawa-Kunita \cite{ishikawa-kunita} etc.), it is quite natural to attempt extending the Dirichlet forms arguments to such cases. This has been done first by Coquio \cite{coquio} when the state space is Euclidean then by Denis \cite{denis} by a time perturbation, see also related works of Privault  \cite{privault}, Albeverio-Kondratiev-R\"ockner \cite{akr}, Ma-R\"ockner \cite{ma-rockner2}.

We shall give a general presentation of Dirichlet forms methods for the Poisson measures in the spirit  of the first approach (Bichteler-Gravereaux-Jacod \cite{bichteler-gravereaux-jacod})  which gives rise to a very similar situation like in Malliavin Calculus : a symmetric semi-group on the Poisson space and a local gradient satisfying the chain rule. With respect to preceding works in this direction we introduce a major simplification due to a new tool {\em the lent particle formula} \cite{bouleau-denis} which gives  the gradient on the Poisson space by a closed formula. Thanks to this representation we obtained with Laurent Denis several results of existence of density \cite{bouleau-denis} \cite{bouleau-denis2} and the method extends to $\mathcal{C}^\infty$ results (forthcoming paper). In this lecture I present the method and the main applications obtained up to now and I expose new results about the regularity of multiple Poisson integrals in connection with the Fock space representation that the Poisson space provides.
It is organised as follows : \\

\noindent{\bf The functional analytic reasoning.}

Dirichlet forms and non-Gaussian Malliavin calculus
---
Poisson random measures
---
Dirichlet form on the Poisson space : the lent particle formula.

\noindent{\bf Practice of the method.}

Other examples
---
Applications to SDE's
---
A useful theorem of Paul L\'evy.

\noindent{\bf Regularity results for multiple Poisson integrals.}

Random Poisson measure and Fock space
---
Decomposition of $\mathbb{D}$ in chaos
---
Density for  $(I_1(g),\ldots,I_n(g^{\otimes n}))$
---
Density for $(I_{n_1}(f_1^{\otimes n_1}),\ldots,I_{n_p}(f_p^{\otimes n_p}))$
---
Other functionals of Poisson integrals
---
Density of $I_n(f)$.

\section{The functional analytic reasoning.}
Let us first introduce the fundamental notions of the theory of local Dirichlet forms.
\subsection{Dirichlet forms and non-Gaussian Malliavin calculus.}
 Let $(X,\cX
,\nu,\bbd,\gamma)$ be a local symmetric Dirichlet structure which
admits a ``carr\'e du champ" operator. This means that $(X,\cX ,\nu )$ is a
measured space, $\nu$ is a $\sigma$-finite positive measure
and the bilinear form
$ e [f,g]=\frac12\int\gamma [f,g]\, d\nu$
is a local  Dirichlet form with domain $\bbd\subset L^2 (\nu )$
and carr\'e du champ $\gamma$ (cf Fukushima-Oshima-Takeda \cite{fukushima-oshima-takeda} in the finite dimensional case and Bouleau-Hirsch \cite{bouleau-hirsch2} in a general setting). The form $e$ is closed in $L^2(\nu)$ and the bilinear operator $\gamma$ satisfies the functional calculus of class $\mathcal{C}^1\cap Lip$:
$$\forall f, g\in\bbd^n,\;\forall F,G \mbox{ of class }\mathcal{C}^1\cap Lip \mbox{ on } \mathbb{R}^n
\quad\gamma[F(f),G(g)]=\sum_{ij}\partial_iF(f)\partial_jG(g)\gamma[f_i,g_j].$$ We write always $\gamma[f]$ for $\gamma[f,f]$ and $e[f]$ for $e[f,f]$.

The space $\bbd$ equipped with the norm $(\|.\|_{L^2(\nu)}^2+e[.,.])^\frac{1}{2}$ is a Hilbert space that we will suppose separable. It is then possible to generate the quadratic differential computations with $\gamma$ by an ordinary differential calculus thanks to the fact that a gradient exists 
 (see Bouleau-Hirsch \cite{bouleau-hirsch2} ex.5.9 p. 242): 
there exist a separable Hilbert space $H$ and a continuous linear
map $D$ from $\bbd$ into $L^2 (X,\nu;H)$ such that

$\bullet$ $\forall u\in \bbd$, $\| D[u ]\|^2_H =\gamma[u]$. 

$\bullet$  If
$F:\R\rightarrow \R$ is Lipschitz  then
$\forall u\in\bbd,\ D[F\circ u]=(F'\circ u )Du,$ where $F'$ is the Lebesgue almost everywhere defined derivative of $F$.

$\bullet$  If $F$ is $\mathcal{C}^1$ (continuously differentiable) and Lipschitz from $\R^d$ into
$\R$ (with $d\in\N$) then
\[ \forall u=(u_1 ,\cdots ,u_d) \in \bbd^d ,\ D[F\circ
u]=\sum_{i=1}^d (\partial_i F \circ u ) D[u_i ].\]

 \noindent In \cite{bouleau-hirsch2} Chap VII we used for $H$ a copy of the space $L^2(\nu)$, but a wide choice is possible depending on convenience.
 
 This differential calculus gives rise to integration by parts formulae as in classical Malliavin calculus. 
 For all $u\in \bbd$ and $v\in\mathcal{D}(a)$ domain of the generator $a$ associated with the Dirichlet structure,
we have
\begin{equation}
\frac{1}{2} \int  \gamma [u,v]d\nu =-\int
ua[v]d\nu.
\end{equation}
The space $\bbd\cap L^\infty$ may be shown to be an algebra,
hence if $u_1 ,u_2\in \bbd\cap L^\infty$ 
\begin{equation}
\frac{1}{2}\int u_2\gamma [u_1,v]d\nu =-\int
u_1 u_2 a[v]d\nu -\frac{1}{2} \int u_1\gamma [u_2,v]
d\nu
\end{equation}
Introducing now the adjoint operator $\delta$ of the gradient $D$, the equality  with $u\in
\bbd$, $U\in \hbox{dom}\; \delta$
\begin{equation}
\int u\delta Ud\nu =\int \langle D[u] ,U\rangle_{H}
d\nu
\end{equation}
provides for $\varphi$ Lipschitz
\begin{equation}
\int \varphi '(u) \langle D[u] ,U\rangle_{H}d\nu
=\int\varphi (u) \delta Ud\nu .
\end{equation}
See \cite{bouleau3} Chap V to VIII and \cite{bouleauMC} for applications of such formulae.\\

But the Dirichlet structures do possess pecular features allowing to show existence of density without using integration by parts arguments. This is based on the following important {\em energy image density property} or (EID):

For each positive integer $d$, we denote by $\cB (\R^d )$ the
Borel $\sigma$-field on $\R^d$ and by $\lambda^d$ the Lebesgue measure
on $(\R^d ,\cB (\R^d ))$.
For $f$ measurable $f_*\nu$ denotes the image of the measure $\nu$ by $f$.

\noindent{\it The Dirichlet structure $(X,\cX
,\nu,\bbd,\gamma)$
is said to satisfy {\rm (EID)} if for any $d$ and for any
$\mathbb{R}^d$-valued function $U$ whose components are in the
domain of the form
$$ U_*[({\det}\gamma[U,U^t])\cdot
\nu ]\ll \lambda^d $$ where {\rm det} denotes the determinant.}

This property is true for any local Dirichlet structure with carr\'e du champ when $d=1$ (cf Bouleau \cite{bouleau1} Thm 5 and Corol 6).  It has been conjectured in 1986 (Bouleau-Hirsch
 \cite{bouleau-hirsch1} p251) that (EID) were true for any  local Dirichlet structure with carr\'e
  du champ. This has been shown for the Wiener space equipped with the Ornstein-Uhlenbeck form and
  for some other structures by Bouleau-Hirsch (cf \cite{bouleau-hirsch2} Chap. II \S 5 and Chap. V example 2.2.4)  but this conjecture
  being at present neither refuted nor proved in full generality, it has to be established in every
  particular setting.  For the Poisson space it has been proved by A. Coquio \cite{coquio} when the intensity measure is the Lebesgue measure on an open set and we obtained with Laurent Denis a rather general condition (\cite{bouleau-denis} Section 2 Thm 2 and Section 4) based on a criterion of Albeverio and R\"ockner \cite{albeverio} and an argument of Song \cite{song}. The new regularity results that are presented here are based on the (EID) property.\\

 Let us first explain the framework of Poisson measures and the notation of the configuration space.
\subsection{Poisson random measures.}
We are
given $(X,\cX ,\nu )$ a measured space. We call it {\it the bottom
space}. 
We assume that $\nu$ is $\sigma$-finite, that for all $x\in X$,
$\{ x\}$ belongs to $\cX$ and that $\nu$ is continuous or diffuse
($\nu(\{x\})=0\;\forall x$).

We consider a random Poisson measure $N$ on $(X,\cX )$ with intensity measure $\nu$. Such a random measure is characterized by the fact that for $A\in\cX$ the random variable $N(A)$ follows a Poisson law with parameter $\nu(A)$ and $N(A_1), \ldots,N(A_n)$ are independent for disjoint $A_i$. Such an object may be constructed on the space of countable sums of Dirac masses on $(X,\cX )$ (the configuration space), by considering first the case where $\nu$ is bounded where the construction is explicit and then proceding by product along a partition of $(X,\cX )$ (see e.g. \cite{bouleau2} or \cite{bouleau3} Chap VI \S 3). We denote by $(\Omega, \mathcal{A}, \mathbb{P})$ the configuration space where $N$ is defined, $\mathcal{A}$ is the $\sigma$-field generated by $N$ and $\mathbb{P}$ its law. The space $(\Omega, \mathcal{A}, \mathbb{P})$ is called {\em the upper space}.

The following density lemma (cf \cite{bouleau-denis}) is the key of several proofs.
 \begin{Le}{\label{lemmeDensite}} For $p\in[1,\infty[$, the set $\{ e^{-N(f)}:\ f\geq 0 , f\in L^1 (\nu )\cap
 L^{\infty} (\nu)\} $ is total in  $L^p (\Om ,\cA ,\bbP )$ and $\{
 e^{iN(f)}:\ f  \in L^1 (\nu )\cap L^{\infty} (\nu)\} $
 is total in $L^p (\Om ,\cA ,\bbP ; \bbC )$.
 \end{Le}
We set $\tN =N-\nu$, then the identity
$ \E[(\tN (f))^2 ] =\int f^2 \, d\nu,$
for $f\in L^1 (\nu )\cap L^2 (\nu)$ can be extended uniquely to
$f\in L^2 (\nu )$ and this permits to define $\tN (f)$ for $f\in L^2
(\nu )$. The Laplace characteristic functional is the basis of all subsequent formulae\begin{equation}\label{211}
\E [e^{i\tN (f)}]=e^{-\int (1-e^{i f}+if)\,d\nu}\ \ \ f\in L^2
(\nu).\end{equation} 
The creation and annihilation operators $\crea$ and $\anni$ well-known in quantum
mechanics (see Meyer \cite{meyer},  Nualart-Vives \cite{nualart-vives}, Picard \cite{picard} etc.) will play a central role for calculus on the configuration space, they are defined  in the following way:
\begin{equation}\label{defin}\begin{array}{l} \forall x,w\in\Om,\ \crea_x (w)=w{\bf 1}_{\{ x\in supp
\, w\}}+(w+\varepsilon_x) {\bf 1}_{\{ x\notin supp \, w\}}\\
\forall x,w\in\Omega,\ \anni_x(w)=w{\bf 1}_{\{ x\notin supp
\, w\}}+(w-\varepsilon_x) {\bf 1}_{\{ x\in supp \, w\}}.
\end{array}\end{equation} 
One can verify that for all $w\in\Om$,
\begin{equation}\label{233}
\crea_x (w)=w\makebox{ and }\anni_x(w)=w-\varepsilon_x \makebox{ for } N_w \makebox{-almost all }x
\end{equation}
and
\begin{equation}\label{234}
\crea_x (w)=w+\varepsilon_x\makebox{ and }\anni_x(w)=w  \makebox{ for } \nu \makebox{-almost
all }x
\end{equation}
We extend these operators to the functionals by setting:
\[ \crea H(w,x)=H(\crea_x w, x)\quad\makebox{ and }\quad\anni H(w,x)=H(\anni_xw,x).\]
This extension recommands to be careful with the order of composition since we have for instance
\begin{equation}\label{ordre}
(\anni\crea H)(x,\omega)=H(x,\crea_x\anni_x\omega)\quad(=H(x,\crea_x\omega)=\crea H)
\end{equation}
It is important to emphasize that since $\nu$ is continuous the two measures $\mathbb{P}\times\nu$ and $\mathbb{P}_N=\mathbb{P}(d\omega)N(\omega)(dx)$ defined on the same sapce $(\Omega\times X,\mathcal{A}\times\mathcal{X})$ are mutually singular. Computation needs to be careful with respect to negligible sets. The next lemma shows that the image of $\bbP\times\nu$ by
$\crea$ is nothing but $\bbP_N$ whose image by $\anni$ is $\bbP\times\nu$ :
\begin{Le}{\label{lem8}} Let $H$ be $\cA\otimes\cX$-measurable
and non negative, then
\[\E \int \crea H d\nu =\E\int H dN\quad\makebox{ and }\quad \E\int\anni HdN=\E\int Hd\nu.\]
\end{Le}
We will encounter also another notion,  sometimes called a ``marked" Poisson measure associated with $N$, which needs here a rigorous construction.

We are still considering $N$ the random Poisson measure on $(X,\cX
,\nu)$ and we are given an auxiliary probability space $(R,\cR
,\rho)$. We construct a random Poisson measure $N\odot\rho$ on
$(X\times R ,\cX\otimes \cR ,\nu\times \rho )$ such that if
$N=\sum_i \varepsilon_{x_i}$ then $N\odot\rho =\sum_i
\varepsilon_{(x_i ,r_i )}$ where $(r_i )$ is a sequence of i.i.d.
random variables independent of $N$ whose common law is
$\rho$.  \\
The construction of $N\odot\rho$ follows line by line the one of
$N$. Let us recall it. We first study the case where $\nu$ is
bounded and we consider the probability space
$ (\N , \cP (\N ), P_{\nu (X)})\times (X,\cX ,\frac{\nu}{\nu
(X)})^{\Ne},$ where $P_{\nu (X)}$ denotes the Poisson law with
parameter $\nu (X)$ and we put
\[ N=\sum_{i=1}^Y \varepsilon_{x_i},\qquad(\mbox{with the convention}\; \sum_1^0=0)\]
where $Y, x_1,\cdots ,x_n ,\cdots $ denote the coordinates maps.
We introduce the probability space
\[ (\hOm,\hcA,\hbbP )=(R,\cR ,\rho )^{\Ne},\]
whose coordinates are denoted by $r_1 ,\cdots ,r_n ,\cdots$. On
the probability space $(\N , \cP (\N ), P_{\nu (X)})\times (X,\cX
,\frac{\nu}{\nu (X)})^{\Ne}\times (\hOm,\hcA,\hbbP )$, we define
the random measure $N\odot\rho =\sum_{i=1}^Y \varepsilon_{(x_i
,r_i )}$. It is a Poisson random measure on $X\times R$ with
intensity measure $\nu\times \rho$. For $f\in L^1 (\nu\times \rho)$
\begin{equation}\label{221}\hE [\int_{X\times R} fdN\odot\rho]=\int_X (\int_R f(x,r)
d\rho(r) )N(dx)\ \ \bbP-a.e.\end{equation} and if $f\in L^2
(\nu\times \rho)$ \begin{equation}\label{222}\hE [(\int_{X\times
R} fdN\odot\rho)^2 ]=(\int_X \int_R f d\rho dN)^2 -\int_X (\int_R
fd\rho)^2 dN +\int_X \int_R f^2 d\rho dN,\end{equation} where $\hE$ stands for the expectation
under the probability $\hbbP$.

If $\nu$ is $\sigma$-finite,  this construction is extended by a
standard product argument. Eventually in all cases, we have
constructed $N$ on $(\Om, \cA ,\bbP)$ and $N\odot\rho$ on $(\Om,
\cA ,\bbP)\times (\hOm, \hcA ,\hbbP)$, it is a random Poisson
measure on $X\times R$ with intensity measure
$\nu\times\rho$, and identities (\ref{221}) and
(\ref{222}) generalize as follows:
 \begin{Pro}{\label{pro6}} Let $F$ be an
$\cA\otimes\cX\otimes\cR$ measurable function such that
$\E\int_{X\times R} F^2 \, d\nu d\rho$ and $\E \int_R (\int_X
|F|d\nu)^2 d\rho$ are both finite then the following relation
holds
\begin{equation}\label{223}
\hE [(\int_{X\times R} FdN\odot\rho)^2 ]=(\int_X \int_R F d\rho
dN)^2 -\int_X (\int_R Fd\rho)^2 dN +\int_X \int_R F^2 d\rho
dN,\end{equation}
in particular if $F$ is such that $\int Fd\rho=0\;\,\mathbb{P}\times\nu$-a.e., then $\hE [(\int_{X\times R} FdN\odot\rho)^2 ]=\int_X \int_R F^2 d\rho
dN.$
\end{Pro} \begin{proof} Approximating first $F$
by a sequence of elementary functions and then introducing a
partition $(B_k)$ of subsets of $X$ of finite $\nu$-measure, this
identity is seen to be a consequence of (\ref{222}).
\end{proof}
Let us take the opportunity to state two formulae that we didn't mention in our preceding articles, and which may be useful in some context. Let $F$ be measurable as in Prop \ref{pro6} and say $0<F\leq 1$ then
\begin{equation}\label{fm}
\hat{\mathbb{E}}\exp{\int\log F\,dN\odot\rho}=\exp{\int(\log\int F\,d\rho})\,dN
\end{equation}
\begin{equation}\label{fmm}
\hat{\mathbb{E}}\int F\,dN\odot\rho=\int(\int F\,d\rho)\,dN
\end{equation}
whose proofs follow the same lines as the construction of $N\odot\rho$ and Prop \ref{pro6}.
\subsection{Dirichlet form on the Poisson space : the lent particle formula.}
Now, after these notions related to the pure probabilistic Poisson space,  we shall assume we have  on the bottom space a Dirichlet structure $(X,\cX
,\nu,\bbd,\gamma)$ as defined in section 2.1. And we attempt to lift up this structure to the Poisson space in a natural manner. This may be done in several ways (see e.g. the introduction of \cite{bouleau-denis}). The method we will follow is not the simplest, we choose it because it enlightens  the role of operators $\crea$ and $\anni$ in the upper gradient. 

First, thanks to (\ref{211}) we obtain the following relation: for all $f\in\bbd$ and all $h\in\cD (a)$,
\begin{equation}\label{212}
\E \left[ e^{i\tN (f)}\left( \tN (a[h])+\frac{i}2 N(\gamma
[f,h])\right)\right]=0.
\end{equation}
This relation and the explicit construction which may be done when $\nu$ is a bounded measure (cf \cite{bouleau2}) suggest a candidate for the generator of the upper structure.

Let us consider the space of test functions $$\cD_0=\mathcal{L}\{e^{i\tN (f)} \mbox{ with }
f\in \cD (a)\cap L^1 ( \nu )\mbox{ et } \gamma[f]\in L^2(\nu)\}.\hspace{1cm}$$ 
and for $U=\sum_p \lambda_p e^{i\tN (f_p)}$ in $\cD_0$, let us put
\begin{equation}{\label{214}}
A_0 [U]=\sum_p \lambda_p e^{i\tN (f_p)}(i\tN (a[f_p ])-\frac12 N
(\gamma [f_p])).\end{equation}
The procedure  to show that
$A_0$  is uniquely defined and is the generator of a Dirichlet form satisfying the hoped properties, has two steps :
first to construct an explicit gradient,
	then to use Friedrichs' property.
\subsubsection{Gradients.}
We will suppose as in section 2.1 that the bottom structure possesses a gradient that we denote from now on $(\cdot)^\flat$. For convenience we assume it satisfies the following properties

$\bullet$ constants belong to $\bbd_{loc}$ (see Bouleau-Hirsch
\cite{bouleau-hirsch2} Chap. I Definition 7.1.3.) 
\begin{equation}\label{232}
1\in \bbd_{loc} \makebox{ which implies }\ \gamma [1]=0 \makebox{
and  } 1^{\flat}=0.
\end{equation}

$\bullet$ $(.)^\flat$ is with values in the orthogonal subspace $L^2_0(R,\mathcal{R},\rho)$ of $1$ in the space $L^2(R,\mathcal{R},\rho)$. This condition is costless since for the gradient only the Hilbert structure of $H$ matters.  From now on we denote this gradient  $(.)^\flat$.   

We take for candidate of  the upper-gradient for $F\in\cD_0$ the pre-gradient
\[ F^\sharp =\int \anni((\crea F)^{\flat})\, dN\odot\rho.\]
where $N\odot\rho$ is the Poisson measure $N$  ``marked'' by $\rho$ as defined in section 2.2.

Let us remark that thanks to Prop 3 and (\ref{232}) we have
\begin{equation}\label{224}
\hE [(\int_{X\times R} \anni((\crea F)^{\flat})\, dN\odot\rho)^2 ]=\int \anni(\gamma(\crea F))  dN \qquad
\bbP\mbox{-}a.e.\end{equation}

For $f\in \cD (a)\cap L^1 ( m ),\, \gamma[f]\in L^2$,  we have $e^{i\tilde{N}(f)}\in\mathcal{D}_0$ and 
\[(e^{i\tN (f)})^\sharp=\int e^{i\tN (f)}(if)^\flat \;dN\odot\rho \]
what yields on $\mathcal{D}_0$:
\begin{equation}\label{236}\hE [F^\sharp \overline{G^\sharp} ]
=\sum_{p,q} \lambda_p \overline{\mu_q} e^{i\tN (f_p -g_q )}N(\gamma(f_p ,g_q)
)
\end{equation}
\subsubsection{Friedrichs' argument.}
This enables us to show that the representation (\ref{214}) does not depend on the expression of $U$ and that $A_0$ is indeed a symmetric negative operator on the dense subspace $\mathcal{D}_0$ of $L^2(\mathbb{P})$ so that Friedrichs' argument applies (see
\cite{bouleau-hirsch2} p.4 or \cite{bouleau3} Lemma III.28 p.48) : it can be extended to a self adjoint operator which may be proved to generate a Dirichlet form with domain $\mathbb{D}$ admitting a carr\'e du champ $\Gamma$  with a gradient extending $(.)^\sharp$. 

It remains only a technical point to verify: the fact that $\mathcal{D}_0$ be dense in $L^2(\mathbb{P})$. This is not obvious because of the condition  $\gamma[f]\in L^2(\nu)$ that we need in $\mathcal{D}_0$ in order $A_0$ take its values in $L^2(\mathbb{P})$. In \cite{bouleau-denis} we called it {\it bottom core hypothesis} (BC), it is not a real constraint in the applications. We can state  (cf \cite{bouleau-denis}) :

\begin{Th}\label{T4}{\em The formula
$$\label{31}
\forall F\in \mathbb{D},\ F^\sharp =\int_{E\times R} \anni((\crea F )^\flat)\,
dN\odot \rho ,$$ extends from $\cD_0$ to $\mathbb{D}$, 
 it is justified by the following  decomposition :
{$$\hspace{-1.5cm} F\in\mathbb{D}\stackrel{\crea-I}{\mapsto} \varepsilon^+F-F\in\underline{\mathbb{D}}\stackrel{\anni((.)^\flat)}{\mapsto}\anni((\varepsilon^+F)^\flat)\in L^2_0(\mathbb{P}_N\times\rho)\stackrel{d(N\odot\rho)}{\mapsto} F^\sharp\in L^2(\mathbb{P}\times\hat{\mathbb{P}})$$}where each operator is continuous on the range of the preceding one 
 and where  $L^2_0
(\bbP_N \times \rho )$ is the closed set of elements  $G$ in $L^2
(\bbP_N \times \rho )$ such that $\int_R G d\rho=0$ $\bbP_N$-a.s.
Furthermore for all $F\in\bbD$
$$\Gamma [F]=\hE (F^\sharp )^2 =\int_E \anni\gamma [\crea F ]\, dN.$$}\end{Th}
This main result --- that we call the {\em lent particle formula} --- implies the validity of a functional calculus  for the obtained Dirichlet structure   $(\Omega,\mathcal{A},\mathbb{P},\mathbb{D},\Gamma)$ on the Poisson space that may be sketched as follows:

Let be $H=\Phi(F_1,\ldots,F_n)$ with $\Phi\in\mathcal{C}^1\cap Lip(\mathbb{R}^n)$ 
and  $F=(F_1,\ldots,F_n)$ with $F_i\in\bbD$, we have :
$$\begin{array}{rrlll}

a)&\gamma[\crea H]&\!\!\!=\sum_{ij}\Phi^\prime_i(\crea F)\Phi^\prime_j(\crea F)\gamma[\crea F_i,\crea F_j]&\mathbb{P}\times\nu\mbox{-a.e.}\\
&&&&\\

b)&\anni\gamma[\crea H]&\!\!\!=\sum_{ij}\Phi^\prime_i(F)\Phi^\prime_j(F)\anni\gamma[\crea F_i,\crea F_j]&\mathbb{P}_N\mbox{-a.e.}\\
&&&&\\

c)&\Gamma[H]=\int\anni\gamma[\crea H]dN&\!\!\!=\sum_{ij}\Phi^\prime_i(F)\Phi^\prime_j(F)\int\anni\gamma[\crea F_i,\crea F_j]dN&\mathbb{P}\mbox{-a.e.}
\end{array}$$
\begin{Rq} {\rm Let $F\in \mathbb{D}$, by the theorem applying formula (\ref{fm}) to $F^\sharp$ gives
$$
\hat{\mathbb{E}}\exp{F^\sharp}=\hat{\mathbb{E}}\exp{\int\anni(\crea F)^\flat\,N\odot\rho}=\exp{\int\left(\log\int\exp{\anni(\crea F)^\flat} d\rho\right) dN}$$
\begin{equation}\label{diese}
=\exp{\int\left(\anni\log\int\exp{(\crea F)^\flat} d\rho\right) dN}
\end{equation} what may yield  the characteristic function of the law of $F^\sharp$ under $\mathbb{P}\times\hat{\mathbb{P}}$ : if we put $\int\exp(iu\crea F)^\flat d\rho=\exp\Psi(u)$ we obtain
$$\mathbb{E}\hat{\mathbb{E}}e^{iuF^\sharp}=\mathbb{E}\exp{\int\anni\Psi(u) dN}.
$$

}\end{Rq}
\subsubsection{Example 1.}
Let  $Y_t$ be a centered L\'evy process with  L\'evy measure $\sigma$ integrating $x^2$ and such that  a local Dirichlet structure may be constructed on $\mathbb{R}\backslash\{0\}$ with carr\'e du champ
$
\gamma[f]=x^2f^{\prime 2}(x).
$
With our notation $(X,\mathcal{X},\nu)=(\mathbb{R}_+\times\mathbb{R}\backslash\{0\},\mbox{Borelian sets},dt\times\sigma)$.

We define the gradient  $\flat$ associated with $\gamma$ by choosing $\xi$  on the auxiliary space $(\hat{\Omega},\hat{\mathcal{A}},\hat{\mathbb{P}})$ such that $\int_0^1\xi(r)dr=0$ and $\int_0^1\xi^2(r)dr=1$ and putting
$f^\flat=xf^\prime(x)\xi(r).
$

The operator $\flat$ acts as a derivation with the chain rule $(\varphi(f))^\flat=\varphi^\prime(f).f^\flat$ (for $\varphi\in\mathcal{C}^1\cap Lip$ or even only Lipschitz).

 $N$ is the  Poisson random measure associated with $Y$ with intensity $dt\times\sigma$ such that $\int_0^th(s)\;dY_s=\int{\bf 1}_{[0,t]}(s)h(s)x\tilde{N}(dsdx)$ for $h\in L^2_{loc}(\mathbb{R}_+)$. (These hypotheses imply  $1+\Delta Y_s\neq 0$ a.s.)
 
  Let us study the existence of density for the pair  $(Y_t ,\mathcal{E}xp(Y )_t)$
where $\mathcal{E}xp(Y )$ is the Dol\'eans exponential of $Y$.

$$\label{exp}\mathcal{E}xp(Y)_t=e^{Y_t}\prod_{s\leq t}(1+\Delta Y_s)e^{-\Delta Y_s}.$$

\noindent$1^0/$ We add a particle $(\alpha ,y)$ i.e. a jump to $Y$ at time $\alpha\leq t$ with size $y$ :
$$\varepsilon^+_{(\alpha,y)}(\mathcal{E}xp(Y)_t)=e^{Y_t+y}\prod_{s\leq t}(1+\Delta Y_s)e^{-\Delta Y_s}(1+y)e^{-y}=\mathcal{E}xp(Y)_t(1+y).$$

\noindent$2^0/$ We compute $\gamma [\varepsilon^+ \mathcal{E}xp(Y)_t](y)=(\mathcal{E}xp(Y)_t)^2 y^2 $.

\noindent$3^0/$ We take back the particle :
\[ \varepsilon^- \gamma [\varepsilon^+ \mathcal{E}xp(Y)_t]=\left( \mathcal{E}xp(Y)_t (1+y)^{-1}\right)^2 y^2\]
we integrate in  $N$  and that gives the upper carr\'e du champ operator (lent particle formula):
\[
\begin{array}{rl} \Gamma [\mathcal{E}xp(Y)_t]&=\int_{[0,t]\times\mathbb{R}}\left( \mathcal{E}xp(Y)_t (1+y)^{-1}\right)^2 y^2N(d\alpha dy)\\
&=\sum_{\alpha \leq t}\left( \mathcal{E}xp(Y)_t (1+\Delta Y_{\alpha})^{-1}\right)^2 \Delta Y_{\alpha}^2.
\end{array}\]
By a similar computation  the matrix $\underline{\underline{\Gamma}}$ of the pair $(Y_t ,\mathcal{E}xp(Y_t ))$ is given by
\[ \underline{\underline{\Gamma}}=\sum_{\alpha \leq t}\left(
                                \begin{array}{cc}
                                  1 &  \mathcal{E}xp(Y)_t (1+\Delta Y_{\alpha})^{-1}\\
                                  \mathcal{E}xp(Y)_t (1+\Delta Y_{\alpha})^{-1} & \left( \mathcal{E}xp(Y)_t (1+\Delta Y_{\alpha})^{-1}\right)^2 \\
                                \end{array}
                              \right)\Delta Y_{\alpha}^2 .
\]
Hence under hypotheses implying (EID) the density of the pair $(Y_t ,\mathcal{E}xp(Y_t ))$ is yielded by the condition
$$\mbox{dim  } \mathcal{L}\left(\left(\begin{array}{c}
1\\
\mathcal{E}xp(Y)_t(1+\Delta Y_\alpha)^{-1}
\end{array}\right)\quad \alpha\in JT\right)=2
$$
where $JT$ denotes the jump times of $Y$ between 0 and $t$. 

Making this in details we obtain

{\it Let $Y$ be a real L\'evy process with infinite L\'evy measure with density dominating a positive continuous function  $\neq0$ near  $0$, then the pair $(Y_t,\mathcal{E}xp(Y)_t)$ possesses a density on  $\mathbb{R}^2$.}

\subsubsection{Example 2.}
Let $Y$ be a real L\'evy process as in the preceding example. 

Let us consider a real c\`adl\`ag process $K$ independent of $Y$ and put $H_s=Y_s+K_s$. Putting $M=\sup_{s\leq t}H_s$ and computing successively $(\varepsilon^+M)$, 
$\gamma[\varepsilon^+M]$ and applying the lent particle formula gives
\begin{Pro}\label{proEx1} If $\sigma(\mathbb{R}\backslash\{0\})=+\infty$ and if $\mathbb{P}[\sup_{s\leq t}H_s=H_0]=0$,  the random variable $\sup_{s\leq t}H_s$ possesses a density.
\end{Pro}
It follows that any real L\'evy process $X$ starting at zero and immediately entering $\mathbb{R}_+^\ast$, whose L\'evy measure dominates a measure $\sigma$ satisfying Hamza's condition (\cite{fukushima-oshima-takeda} p105) and infinite, is such that $\sup_{s\leq t}X_s$ has a density.

\subsubsection{Example 3. L\'evy's stochastic area.}
This example will show that the method can detect densities even when both the Malliavin matrix  is non invertible and the  L\'evy measure is singular.

Let $X(t)=(X_1(t), X_2(t))$ be a  L\'evy process with values in $\mathbb{R}^2$ with  L\'evy measure $\sigma$. We suppose  that the hypotheses of the method are fulfilled, we shall explicit this later on.

Let us consider first a general gradient on the bottom space :
$$f^\flat=f_1^\prime\xi_1+f_2^\prime\xi_2$$
{ where $f'_i =\frac{\partial f}{\partial x_i}, $ and
$\xi_1$, $\xi_2$ are functions defined on $\R^2 \times R$ which
satisfy: $\int_R \xi_1 (\cdot , r)\rho (dr)=\int_R \xi_2 (\cdot ,r
)\rho (dr)=0$, $\int_R \xi_1^2 (x_1 ,x_2 ,r)\rho
(dr)=\alpha_{11}(x_1,x_2)$, $\int_R \xi_1 (x_1 ,x_2 ,r)\xi_2 (x_1
,x_2 ,r)\rho (dr)=\alpha_{12}(x_1,x_2)$, $\int_R \xi_2^2 (x_1 ,x_2
)\rho (dr)=\alpha_{22}(x_1,x_2)$}, so that
$$\gamma[f]=\alpha_{11}f^{\prime 2}_1+2\alpha_{12}f^\prime_1f^\prime_2+\alpha_{22}f_2^{\prime2}.$$
Let us consider the following vector involving L\'evy's stochastic area
$$V=(X_1(t),X_2(t),\int_0^tX_1(s_-)dX_2(s)-\int_0^tX_2(s_-)dX_1(s)).$$
We have  for $0<\alpha < t$ and $x=(x_1 ,x_2
)\in\R^2,$
$$\varepsilon^+_{(\alpha ,x)}V=V+(x_1,x_2, X_1(\alpha_-)x_2+x_1(X_2(t)-X_2(\alpha))-X_2(\alpha_-)x_1-x_2(X_1(t)-X_1(\alpha))$$
$$\qquad=V+(x_1,x_2,x_1(X_2(t)-2X_2(\alpha))-x_2(X_1(t)-2X_1(\alpha)))$$
because $\varepsilon^+V$ is defined
$\mathbb{P}\times\nu{ \times d\alpha}$-a.e. and $\nu\times d\alpha$ is
diffuse, so
$$(\varepsilon^+V)^\flat=(\xi_1,\xi_2,\xi_1(X_2(t)-2X_2(\alpha))-\xi_2(X_1(t)-2X_1(\alpha)))$$
and
$$
\gamma[\varepsilon^+V]=\left(
\begin{array}{ccc}
\alpha_{11}&\alpha_{12}&A\alpha_{11}-B\alpha_{12}\\
\alpha_{12}&\alpha_{22}&A\alpha_{12}-B\alpha_{22}\\
A\alpha_{11}-B\alpha_{12}&A\alpha_{12}-B\alpha_{22}&A^2\alpha_{11}-2AB\alpha_{12}+B^2\alpha_{22}
\end{array}
\right)
$$
denoting
$A=(X_2(t)-2X_2(\alpha))$ and $B=(X_1(t)-2X_1(\alpha))$.

 This yields
$$\varepsilon^-A=X_2(t)-\Delta X_2(\alpha)-2X_2(\alpha_-)\qquad\mbox{let us denote it }\tilde{A}$$
$$\varepsilon^-B=X_1(t)-\Delta X_1(\alpha)-2X_1(\alpha_-)\qquad\mbox{let us denote it }\tilde{B}$$
and eventually
$$
\Gamma[V]=\sum_{\alpha\leq t}\left(
\begin{array}{ccc}
\alpha_{11}(\Delta X_\alpha)&\alpha_{12}(\Delta X_\alpha)&\tilde{A}\alpha_{11}(\Delta X_\alpha)-\tilde{B}\alpha_{12}(\Delta X_\alpha)\\
\sim&\alpha_{22}(\Delta X_\alpha)&\tilde{A}\alpha_{12}(\Delta X_\alpha)-\tilde{B}\alpha_{22}(\Delta X_\alpha)\\
\sim&\sim&\tilde{A}^2\alpha_{11}(\Delta X_\alpha)-2\tilde{A}\tilde{B}\alpha_{12}(\Delta X_\alpha)+\tilde{B}^2\alpha_{22}(\Delta X_\alpha)
\end{array}
\right)
$$ the symbol $\sim$ denoting the symmetry of the matrix.

Considering the case $\alpha_{12}=0$  let us take the  L\'evy measure of $(X_1,X_2)$ expressed in polar coordinates as
$$\nu(d\rho,d\theta)=g(\theta)d\theta.1_{]0,1[}(\rho)\frac{d\rho}{\rho}$$
with $g$ locally bounded and such that it dominates a continuous and positive function near $0$.
Then $V=(X_1(t),X_2(t), \int_0^tX_1(s_-)dX_2(s)-\int_0^tX_2(s_-)dX_1(s))$ has a density (and condition (0.4) of \cite{cancelier-chemin} or of \cite{picard} prop1.1 are not fulfilled).

Considering now the case  $\xi_2=\lambda(x_1,x_2)\xi_1$ 
which applies to $V=(X_1(t),[X_1]_t, \int_0^tX_1(s_-)d[X_1](s)-\int_0^t[X_1](s_-)dX_1(s))$.

 The L\'evy measure of $(X_1,[X_1])$ is carried by the curve  $x_2=x_1^2$. We have $\lambda(x_1,x_2)=2x_1$. We arrive to the sufficient condition :
{\it $V$ has a  density as soon as the L\'evy measure of $X_1$ is infinite and satisfies hypotheses for {\rm(BC)} and {\rm (EID)}}. (cf \cite{bouleau-denis} and \cite{bouleau-denis2}).

\section{Practice of the method.}
\subsubsection{Computation with the lent particle formula.}
The presence of operators $\crea$ and $\anni$ in the lent particle formula (Thm \ref{T4}) which exchange the mutually singular measures $\mathbb{P}_N$ and $\mathbb{P}\times\nu$, requires to be more careful than in the usual stochastic calculus where all is defined $\mathbb{P}$-a.s. We make some remarks and give some examples to help the reader to become familiar with this tool.

\subsubsection{The lent particle formula extends to $\mathbb{D}_{loc}$.} The space $\mathbb{D}_{loc}$ is a remarkable specific feature of local Dirichlet forms with carr\'e du champ : the carr\'e du champ operator extends to functions locally -- in a measurable sense -- in $\mathbb{D}$ (cf \cite{bouleau-hirsch2} Chap I \S7.1).

\noindent{\em We denote $\mathbb{D}_{loc}$ the set of applications $F: \Omega\mapsto \mathbb{R}$ such that there exists a sequence $\Omega_n\in\mathcal{A}$ such that $\cup_n\Omega_n=\Omega$ and $\exists F_n\in\mathbb{D}$ with $F=F_n$ on $\Omega_n$.}

The fact that (EID) is always true for $d=1$ (cf \cite{bouleau1}) shows that, for $F\in\mathbb{D}_{loc},$ $\Gamma[F]$ is uniquely defined and may be evaluated by $\Gamma[F_n]$ on $\Omega_n$. The operator $\sharp$ extends to $\mathbb{D}_{loc}$ by putting $F^\sharp=F_n^\sharp$ on $\Omega_n$. For $F$ in $\mathbb{D}_{loc}$, the formulae
$$F^\sharp=\int\anni((\crea F)^\flat)\,dN\odot\rho\qquad \Gamma[F]=\int\anni(\gamma[\crea F])dN$$
resume a computation done on each $\Omega_n$.
\subsubsection{Negligible sets.}
As it was recalled above at the beginning of section 3,
it is recommended to write down the negligible sets at each equality e.g.
$$\begin{array}{rcll}
\crea(\tilde{N}f)&=&\tilde{N}f+f&\mathbb{P}\times\nu\mbox{-a.e.}\\
\anni(\tilde{N}f)&=&\tilde{N}f-f&\mathbb{P}_N\mbox{-a.e.}\\
\crea(e^{i\tilde{N}f}g)&=&e^{i\tilde{N}f}e^{if}g&\mathbb{P}\times\nu\mbox{-a.e.}\\
\anni(e^{i\tilde{N}f}g)&=&e^{i\tilde{N}f}e^{-if}g&\mathbb{P}_N\mbox{-a.e.}
\end{array}
$$
\begin{Rq}\label{R6} {\rm Let us observe that if $H(\omega, x)=G(\omega)g(x)$ where $G$ is defined $\mathbb{P}$-a.s. and $g$ $\nu$-a.e.
then $H$ belongs necessarily to a single class $\mathbb{P}_N$-a.e. So that we may apply to $H$ both operators $\crea$ and $\anni$ without ambiguity. This will be used further about multiple Poisson integrals.}\end{Rq}

\subsubsection{A simplified sufficient condition.}
Theorem \ref{T4} gives a method for obtaining $\Gamma[F]$ for $F\in \mathbb{D}$ or $F\in\mathbb{D}^n $, then with the hypotheses giving (EID) it suffices to prove $\mbox{det }\Gamma[F]>0$ $\mathbb{P}$-a.s.  to assert that $F$ has a density on $\mathbb{R}^n$.
 Let us mention a stronger condition which may be also useful in some applications. By the following lemma that we leave to the reader

\begin{Le}\label{determin}{\em 
Let $M_\alpha$ be random symmetric positive matrices and $\mu(d\alpha)$ a random positive measure. Then
$\{\mbox{\rm det}\int M_\alpha \mu(d\alpha)=0\}\subset\{\int{\mbox{\rm det}}M_\alpha\mu(d\alpha)=0\},$}\end{Le}

\noindent it is enough to have $\int\mbox{det }\varepsilon^-(\gamma[\varepsilon^+F])dN>0$ $\mathbb{P}$-a.s. hence enough that $\mbox{det }\varepsilon^-(\gamma[\varepsilon^+F])$ be $>0$ $\mathbb{P}_N$-a.e. We obtain, by lemma 2, that a sufficient condition for the density of $F$ is $\mbox{det }\gamma[\varepsilon^+F]>0\;$ $\mathbb{P}\times\nu\times dt$-a.e. (or  equivalently that the components of the vector $(\varepsilon^+F)^\flat$ be $\mathbb{P}\times\nu\times dt$-a.e. linearly independent in $L^2(\rho)$ ).

\subsubsection{The energy image density property} (EID).
We gave  in Bouleau-Denis \cite{bouleau-denis} general conditions on the bottom structure $(X,\mathcal{X},\nu,\bbd,\gamma)$ to satisfy (EID) and for this property to be lifted up to the upper space $(\Omega,\mathcal{A},\mathbb{P},\mathbb{D},\Gamma)$. Here are these conditions in a simplified form:

\begin{Pro}{\em Suppose $(X,\mathcal{X},\nu)=(\mathbb{R}^d,\mathcal{B}(\mathbb{R}^d),k(x)dx)$ with $k$ continuous on an open set of full Lebesgue measure and suppose the carr\'e du champ operator is defined on the test functions $\mathcal{C}^\infty_K$  infinitely differentiable with compact support by the formula
\begin{equation}\label{occ}\sum_{ij}\xi_{ij}(x)\partial_if(x)\partial_j f(x)\end{equation}
where $\xi$ is locally bounded and locally elliptic i.e. for every compact $K$ there are constants $C_K<\infty$ and $c_K>0$ such that
$\forall x\in K,\,\forall c\in\mathbb{R}^d\quad C_K|c|^2\geq \sum_{i,j=1}^d \xi_{ij} (x)c_i c_j \geq
c_K |c|^2 $, then the bilinear form
\begin{equation}\label{truc} e[u,v]=\frac12 \int_{\R^r}\sum_{i,j}
\xi_{ij}(x) \partial_i u(x)\partial_j v (x) k(x)\, dx .\end{equation}
defined on $\mathcal{C}^\infty_K$ is closable and its closure defines a Dirichlet form $(e,\bbd)$ with carr\'e du champ given by {\rm(\ref{occ})}, and this structure satisfies {\rm(EID)} and {\rm(BC)}.}\end{Pro}

It is useful for many examples to remark that the preceding case allows to extend (EID) and (BC) to situations where $\nu$ is singular w.r. to Lebesgue measure.

Let $(\mathbb{R}^p\backslash\{0\},\mathcal{B}(\mathbb{R}^p\backslash\{0\}), \nu, \mathbf{d}, \gamma)$ be a Dirichlet structure on $\mathbb{R}^p\backslash\{0\}$ satisfying (EID).
Let $U: \mathbb{R}^p\backslash\{0\}\mapsto\mathbb{R}^q\backslash\{0\}$ be an injective map ($p<q$) such that $U\in \mathbf{d}^q$ . Then $U_*\nu$ is $\sigma$-finite. If we put
$$\begin{array}{rl}
\mathbf{d}_U&=\{\varphi\in L^2(U_*\nu):\varphi\circ U\in\mathbf{d}\}\\
e_U[\varphi]&=e[\varphi\circ U]\\
\gamma_U[\varphi]&=\frac{d\;U_*(\gamma[\varphi\circ U].\nu)}{d\;U_*\nu}
\end{array}
$$ then
the term $(\mathbb{R}^q\backslash\{0\},\mathcal{B}(\mathbb{R}^q\backslash\{0\}), U_*\nu, \mathbf{d}_U, \gamma_U)$ is a Dirichlet structure satisfying (EID).  Additional regularity assumptions make $U$  transport also property (BC).

Now it is possible to lift up (EID) from the bottom to the upper space if two conditions are fulfilled. First to be able to share the bottom space on a partition of sets of finite $\nu$-measure. Second that the obtained Dirichlet structures are such that any finite product satisfies (EID). The precise formulation is given in Bouleau-Denis \cite{bouleau-denis} Section 4. This covers all cases encountered in practice.
\subsection{Other examples.}
\subsubsection{Example 4. Nearest point of the origin.} This example shows the quickness of the method which has, in some sense, to be paid by the care to put on negligible sets.

Let us take for the bottom space $(\mathbb{R}^d,\mathcal{B}(\mathbb{R}^d),\nu,\bbd,\gamma)$ satisfying (BC), assuming  the identity map $j$ on $\mathbb{R}^d$  belong to $\bbd^d$ and $\gamma[|j|]>0$, the measure $\nu$ being infinite, possibly carried by a surface or a curve.
Let us consider the functional $H$ defined on $(\Omega,\mathcal{A}, \mathbb{P})$
$$H(\omega)=\inf_{x\in\mbox{supp}(\omega)}|x|.$$
The $\inf$ is reached because the measure $\nu$ is  $\sigma$-finite. We have 
$$\crea_xH=|x|\wedge H\quad\mathbb{P}\times\nu\mbox{-a.e.}$$ We will suppose that the measure $\nu$ does not charge the level surfaces of $|x|$ i.e. the spheres centered at O. Then for fixed $\omega$, $x\mapsto\crea_xH$ belongs to $\bbd$ and we have
$$(\crea_xH)^\flat=(|j|)^\flat1_{|j|\leq H}=(|j|)^\flat1_{|j|<H}\qquad \mathbb{P}\times\nu\times\rho\mbox{-a.e.}$$ The two functionals $1_{|j|\leq H}$ and $1_{|j|< H}$  equal $\mathbb{P}\times\nu$-a.e. do have the same image by $\anni$ $\mathbb{P}_N$-a.e.
$$1_{|x|\leq H(\anni_x\omega)}=1_{|x|<H(\anni_x\omega)}\quad \mathbb{P}_N\mbox{-a.e.}$$
and the lent particle formula gives
$$\Gamma[H]=\int\gamma[|j|](x)1_{|x|\leq H(\anni_x\omega)}\,N(\omega,dx)=\int\gamma[|j|](x)1_{|x|< H(\anni_x\omega)}\,N(\omega,dx)\quad \mathbb{P}\mbox{-a.s.}$$
Now this integral is easily seen to be equal to  $\gamma[|j|](x_0(\omega))$ where $x_0$ is the $\mathbb{P}$-a.s-unique point achieving the minimum of the distance of the support of $\omega$ to the origin. Thus we obtain the quite natural result that as soon as $\nu$ doesn't charge the spheres, $H$ possesses a density.

As in several other examples, the result could be extended to the case where $\nu(\mathbb{R}^d)$ be finite by conditioning by the event $\{N(\mathbb{R}^d)\geq 1\}$.
\subsubsection{Example 5. Gas of Brownian particles.} This is an extension of the preceding example to infinite dimensional setting. We consider a gas of Brownian particles in $\mathbb{R}^3$. Each particle is independent, the initial positions are distributed in $\mathbb{R}^3$ along a Poisson measure with uniform intensity. We study the lowest distance of a particle to the origin during the time interval $[0,1]$.

\noindent A) Let us begin with some properties of extrema on the Wiener space. Let be given a Brownian motion $B_t=(B_t^1,B_t^2,B_t^3)$ starting at zero, the Wiener space being endowed with the Ornstein-Uhlenbeck structure. We adopt --- only in this paragraph A) --- the following notation for this structure $(W,\mathcal{W},m,\mathbb{D},\Gamma)$ and we use the gradient with values in $L^2(\hat{W},\hat{\mathcal{W}},\hat{m})$ where $(\hat{W},\hat{\mathcal{W}},\hat{m})$ is a copy of $(W,\mathcal{W},m)$ defined by
$$(\int_0^1f(t)\cdot dB_t)^\sharp=\int_0^1f(t)\cdot d\hat{B}_t\qquad \forall f=(f_1,f_2,f_3)\in L^2([0,1]).$$
If $x\in\mathbb{R}^3$ is fixed and $\neq0$, the random variate $$K(w)=\inf_{t\in[0,1]}|x+B_t|$$ is strictly positive and in $\mathbb{D}$, by the argument developped by Nualart-Vives \cite{nualart-vives2}, using the fact that the set of Brownian paths which reach several times the minimum is negligible, we obtain
$$K^\sharp(w,\hat{w})=\frac{(x_1+B^1_{T(w)}(w))\hat{B}^1_{T(w)}(\hat{w})+
(x_2+B^2_{T(w)}(w))\hat{B}^2_{T(w)}(\hat{w})+(x_3+B^3_{T(w)}(w))\hat{B}^3_{T(w)}(\hat{w})}{|x+B_{T(w)}(w)|}$$ where $T(w)=\inf\{t\in[0,1] : |x+B_t(w)|=K(w)\}$.

It follows that
$$\Gamma[K]=\hat{\mathbb{E}}[(K^\sharp)^2]=T>0\quad\mbox{a.s. if }x\neq 0.$$
B) Let us come back to our usual notation. For the bottom space we take $(X,\mathcal{X},\nu)=(\mathbb{R}^3\times W,\mathcal{B}(\mathbb{R}^3)\times\mathcal{W}, \lambda^3\times m)$ where $\lambda^3$ is the 3-dimensional Lebesgue measure, that we equip with the product Dirichlet structure of the zero form on $\mathbb{R}^3$ and the O-U-form on the Wiener space. The structure $(X,\mathcal{X},\nu,\bbd,\gamma)$ is thus naturally endowed with a gradient induced by the gradient used in part A) and that we denote now $\flat$ as usual, it is with values in $L^2(\hat{m})$. The hypothesis (BC) is fulfilled.

We construct the upper structure $(\Omega, \mathcal{A},\mathbb{P},\mathbb{D},\Gamma)$ which describes a gas of Brownian particles. We denote $(x,w)$ the current point of $X$ and we consider the functional
$$H(\omega)=\inf_{\begin{array}{c}
t\in[0,1]\\
(x,w)\in\mbox{supp }\omega
\end{array}}|x+B_t(w)|.$$ We apply the lent particle method :
$$\crea_{(x,w)}H=(\inf_{t\in[0,1]}|x+B_t(w)|)\wedge H$$
Here the measure $\lambda^3\times m$ does not charge the level sets of $(\inf_{t\in[0,1]}|x+B_t(w)|)$ and we have
$$\begin{array}{rl}
(\crea H)^\flat&=(\inf_{t\in[0,1]}|x+B_t(w)|)^\flat1_{\{(\inf_{t\in[0,1]}|x+B_t(w)|)\leq H\}}\\
&=(\inf_{t\in[0,1]}|x+B_t(w)|)^\flat1_{\{(\inf_{t\in[0,1]}|x+B_t(w)|)< H\}}\qquad \mathbb{P}\times\nu\times\hat{m}\mbox{-a.e.}
\end{array}$$ what gives putting $\eta(x,w)=\inf_{t\in[0,1]}|x+B_t(w)|$ and $a=(x,w)$
$$\Gamma[H]=\int\gamma[\eta]1_{\{\eta(a)\leq H(\anni_{a}\omega)\}}N(\omega,da)$$
and this is equal to $\gamma[\eta]$ taken on the unique Brownian particle which yields the minimum.

Since by the part A) above this quantity is strictly positive, we can conclude that $H$ possesses a density.

The argument extends to the case where the point taken as origin is itself moving deterministically or as an independent diffusion process.

\subsubsection{Example 6. Integral of a L\'evy process.} let $Y_t$ be a L\'evy process with values in $\mathbb{R}^d$ with our usual hypotheses that the L\'evy measure $\sigma$ carries a Dirichlet form such that hold (BC) and (EID). Let us suppose in addition for simplicity that $\sigma$ integrates $|x|^2$, that $Y$ is centered without Brownian part and that the coordinate maps $x_i$ are in $\bbd$ with $\gamma[x_i,x_j]=x_ix_j\delta_{ij}$.

Let be $g\in\mathcal{C}^1\cap Lip$ from $\mathbb{R}^d$ into itself and let us consider the $d$-dimensional functional
$$H=\int_0^1g(Y_t)dt$$ which writes also $H=\int_0^1g(\int 1_{[0,t]}(s)y\,\tilde{N}(dyds))dt$ if $N$ denotes the Poisson measure associated with $(Y)$. For $0\leq \alpha\leq 1$
$$\begin{array}{rcll}
\crea_{(\alpha,y)}H&=&\int_0^\alpha g(Y_t)dt+\int_\alpha^1g(Y_s+y)ds\quad &\mathbb{P}\times\nu\mbox{-a.e.}\\
(\crea_{(\alpha,y)}H)^\flat&=&\int_\alpha^1Dg(Y_s+y)ds\cdot j^\flat(y)\quad &\mathbb{P}\times\nu\times\rho\mbox{-a.e.}
\end{array}
$$ where $Dg$ is the Jacobian matrix of $g$ and $j$ the identity map on $\mathbb{R}^d$. Then
$$\begin{array}{rcll}
\anni(\crea_{(\alpha,y)}H)^\flat&=&\int_\alpha^1Dg(Y_s)ds\cdot j^\flat(y)\quad& \mathbb{P}_N\times\rho\mbox{-a.e.}\\
\Gamma[H]&=&\sum_{\alpha\leq1}\int_\alpha^1Dg(Y_s)ds\;\gamma[j,j^t](\Delta Y_\alpha)\,\int_\alpha^1D^tg(Y_s)ds\quad&\mathbb{P}\mbox{-a.s.}
\end{array}$$the matrix $\gamma[j,j^t](\Delta Y_\alpha)$ is the $d\times d$-matrix whose diagonal is composed of the squares of the jumps $((\Delta Y^1_\alpha)^2,\ldots,(\Delta Y_\alpha^d)^2)$. If the images by the coordinate mappings of the L\'evy measure are infinite, and if the Jacobian matrix $Dg$ is regular, then $H$ has a density on $\mathbb{R}^d$.
See \cite{lifshitz} and \cite{bertoin} for related results.

\subsubsection{Example 7.  Generalized Ornstein-Uhlenbeck processes.} Let $(\xi,\eta)$ be a 2-dimensional L\'evy process starting from (0,0). The process
$$X_t=e^{\xi_t}(x+\int_0^te^{-\xi_s}\,d\eta_s)\qquad t\geq0\quad x\in\mathbb{R}$$ is a homogeneous Markov process called generalized O-U process driven by $(\xi,\eta)$ (cf \cite{carmona}). It is possible to see by the classical Malliavin calculus that if $(\xi,\eta)$  possesses a Brownian part then $X_t$ has a density. We exclude this case now and suppose that the L\'evy measure   carries a Dirichlet form satisfying (BC) in order to apply the method (without care of (EID) because $X_t$ is one dimensional).

Let us begin by computing $\Gamma[X_t]$ by the lent particle method.

\noindent Let  $(\alpha, \xi, \eta)$ denote the current point of $X=\mathbb{R}_+\times\mathbb{R}\times\mathbb{R}$,
$$\crea_{(\alpha, \xi, \eta)}\xi_t=\xi_t+\xi1_{\alpha\leq t}\qquad\crea_{(\alpha, \xi, \eta)}\xi_{t-}=\xi_{t-}+\xi1_{\alpha<t}$$ 
$$
\begin{array}{rcl}
\crea_{(\alpha, \xi, \eta)}X_t&=&
e^{\xi_t+\xi1_{\alpha\leq t}}\left[x+\int_{[0,t]}e^{-(\xi_{s-}+\xi1_{\alpha<s})}\,d(\eta_s+\eta1_{\alpha\leq s})\right]\\
(\crea X_t)^\flat&=&
e^{\xi_t+\xi1_{\alpha\leq t}}\left[x\xi^\flat+\int_{[0,\alpha]}e^{-(\xi_{s-}+\xi1_{\alpha<s})}d\eta_s\xi^\flat+e^{-\xi_{s-}}(\eta^\flat+\eta\xi^\flat)\right]\\
\anni(\crea X_t)^\flat&=&
e^{\xi_t}\left[x\xi^\flat+\int_{[0,\alpha]}e^{-\xi_{s-}}d\eta_s\xi^\flat+e^{-\xi_{\alpha-}}\eta^\flat\right]
\end{array}$$
Let $j$ be the identity map on $\mathbb{R}^2$, so that $j^\flat=(\xi^\flat,\eta^\flat)$ we obtain
$$\Gamma[X_t]=e^{2\xi_t}\int_0^t(x+\!\!\int_{[0,\alpha]}\!\!e^{-\xi_{s-}}d\eta_s\quad e^{-\xi_{\alpha-}})\gamma[j,j^t]
\left(\begin{array}{c}
x+\int_0^\alpha e^{-\xi_{s-}}d\eta_s\\
e^{-\xi_{\alpha-}}
\end{array}\right)
N(d\alpha d\xi d\eta)$$
now putting $V(\omega,x,\alpha,\xi,\eta)=\left(\begin{array}{c}
x+\int_0^\alpha e^{-\xi_{s-}}d\eta_s\\
e^{-\xi_{\alpha-}}
\end{array}\right)$ this writes
$$\Gamma[X_t]=e^{2\xi_t}\sum_{\alpha\leq t}V^t\gamma[j,j^t](\Delta\xi_\alpha,\Delta\eta_\alpha)V$$
the sum being taken on the jump times of the process $(\xi_t,\eta_t)$. Starting from this relation we discuss several cases :

1) First case $\det\gamma[j,j^t]>0$.

\noindent Since (EID) does not matter $X_t$ being real valued, the only condition is that the L\'evy measure $\sigma$ of $(\xi_t,\eta_t)$ be infinite and carry a local Dirichlet structure  $(\mathbb{R}^2\backslash\{(0,0)\},\mathcal{B}(\mathbb{R}^2\backslash\{(0,0)\}),\sigma,\bbd,\gamma)$ satisfying (BC) and such that $j\in\bbd_{loc}$  and $\det\gamma[j,j^t]>0$ $\sigma$-a.e.

\noindent No necessary and sufficient condition is known for this which would extend the Hamza condition to dimension 2, but we see by Prop 7 that this will be fulfilled as soon as $\sigma$ has a continuous density.

2) The case where $\xi_t$ and $\eta_t$ are independent.

\noindent The measure $\sigma$ is carried by the coordinate axes,
$\gamma[j,j^t](a,b)=\left(
\begin{array}{cc}
\varphi(a)1_{b=0}&0\\
0&\psi(b)1_{a=0}
\end{array}
\right)
$ and 
$$\Gamma[X_t]= e^{2\xi_t}\sum_{\alpha\leq t}\left[(x+\int_{[0,\alpha]} e^{-\xi_{s-}}d\eta_s)^2\varphi(\Delta\xi_\alpha)1_{\Delta\eta_\alpha=0}+
e^{-2\xi_{\alpha-}}\psi(\Delta\eta_\alpha)1_{\Delta\xi_\alpha=0}\right].$$
If the L\'evy measure of $(\eta_t)$ is infinite and if $\psi>0$ then $X_t$ has a density. 

\noindent If the L\'evy measure of $(\eta_t)$ is finite, then if $\varphi>0$ and if the L\'evy measure of $(\xi_t)$ is infinite $X_t$ has a density as soon as $x+\int_{[0,\alpha]} e^{-\xi_{s-}}d\eta_s$ does not vanish for little $\alpha$ hence as soon as $x\neq 0$.

In this case of independence it is also possible to use the representation in law (cf \cite{carmona} Thm 3.1)
$$X_t\stackrel{d}{=} e^{\xi_t}x+\int_0^te^{\xi_{s-}}\,d\eta_s$$ applying the lent particle method to the right hand side. That gives a little faster the same conclusion.

3) The case where $\sigma$ is carried by a curve.

\noindent We sketch only this case which involves a parametrization. Let  $M$ be a Poisson measure on $\mathbb{R}_+$ with $\sigma$-finite intensity measure $m$ and let be given a map $\Phi:u\mapsto(f(u),g(u))\in\mathbb{R}^2\backslash\{(0,0)\}$ such that we obtain our L\'evy process by image :
$$\Phi_*m=\sigma\quad\Phi_*M=N$$ $\Phi$ being injective from $\mathbb{R}_+$ into $\mathbb{R}^2\backslash\{(0,0)\}$ and such that $\lim_{a\rightarrow\infty} \Phi(a)=(0,0)$. On $\mathbb{R}_+$ we start with a Dirichlet structure  $(\mathbb{R}_+,\mathcal{B}(\mathbb{R}_+), m,\tilde{\bbd},\tilde{\gamma})$. We assume the identity $J\in\tilde{\bbd}_{loc}$ and $f$ and $g$ of class $\mathcal{C}^1\cap Lip$.

We have
$\gamma[j,j^t]=\left(\begin{array}{c}
f^{\prime 2}\quad f^\prime g^\prime\\
f^\prime g^\prime\quad g^{\prime 2}
\end{array}
\right)\tilde{\gamma}[J]
$ and we obtain $$\Gamma[X_t]=
e^{\xi_t}\int\left[f^\prime(a)(x+\int_0^\alpha e^{-\xi_{s-}}d\eta_s)+g^\prime(a)e^{-\xi_{\alpha-}}\right]^2\tilde{\gamma}[J](a)\,M(d\alpha da).$$
Let us suppose the L\'evy measure of $(\xi_t,\eta_t)$ infinite, i.e. $m$ infinite, 
and $\lim_{a\rightarrow\infty}(f^\prime(a),g^\prime(a))$ exist and be equal to $(v_1,v_2)\neq(0,0)$.

 $\Gamma[X_t]=0$ for some $\omega$ would imply $v_1x+v_2=0$ what can be realized only for one value of $x$. 
The reasoning may then be improved by considering the behaviour at the neighborhood of another time $\alpha_0$. 

\subsubsection{Example 8. Interaction potential.} Several forms of interaction potential are encountered in physics for an infinite system of interacting particles: $\exp\{-\beta\sum_{ij}\Psi(X_i-X_j)\}$ , $\alpha\beta^n\prod_{ij}g(|X_i-X_j|)$ or $\exp\{\sum_{ij}a(X_i)a(X_j)b(|X_i-X_j|)\}$ etc.

Let us consider the functional $\Phi=\int\varphi(x)\varphi(y)\psi(|x-y|^2)N(dx)N(dy)$ where the functions $\varphi$ and $\psi$ are regular, $\psi(0)=0$, $N$ being a random Poisson measure on $\mathbb{R}^3$.

After computing as usual $\crea_x\Phi$, $(\crea\Phi)^\flat$ and $\anni(\crea \Phi)^\flat$, the lent particle theorem gives
$$\Gamma[\Phi]=
\int V(x)^t\gamma[j,j^t]V(x)\; N(dx)$$ where $j$ is the identity on $\mathbb{R}^3$ and
$V(x)$ is the column vector $$V(x)=\int \left(2\varphi(\alpha)\psi(|x-\alpha]^2)\nabla\varphi(x)+
4\varphi(x)\varphi(\alpha)\psi^\prime(|x-\alpha|^2)(x-\alpha)\right)N(dx).$$
If the bottom structure is such that $\gamma[j,j^t]$ may be chosen to be the identity matrix, we have
\begin{equation}\label{potinter}
\Gamma[\Phi]=\int\left|\int F(x,y)N(dy)\right|^2N(dx)
\end{equation}
with $F=[2\psi(|x-y|^2)\nabla\varphi(x)+4\varphi(x)\psi^\prime(|x-y|^2)(x-y)]\varphi(y)$.

In order to study the positivity of $\Gamma[\Phi]$, we will use the following lemma (due to Paul L\'evy 1931) on which we will come back in the next section.
\begin{Le}\label{lemme de levy}{\em Let $f$ be measurable  on the bottom space such that $\int |f|\wedge 1\,d\nu<+\infty$.

If $\nu\{f\neq 0\}=+\infty$ then the law of $N(f)$ is continuous.}
\end{Le}
That gives us the following result
\begin{Pro}{\em If $F$ is such that
{\rm(i)} $\exists G\in L^1(\nu) : |F(x,y)|\leq G(y)$,
{\rm(ii)} $\forall y \quad x\mapsto F(x,y)$ is continuous,
{\rm(iii)} $\forall x\quad\nu\{F(x,.)\}=+\infty$,
\noindent then $(\int F(x,y)\,N(dy)\neq 0)\quad \mathbb{P}$-a.s.}
\end{Pro}
\noindent{\it Proof.} For $\omega$ outside a negligible set $F(x,.)$ is bounded in modulus by an integrable function for $N(\omega, dy)$, hence $x\mapsto\int F(x,y)N(dy)$ is continuous by dominated convergence, hence the set $\{x : \int F(x,y)N(dy)\neq 0\}$ is open; by the property (iii) and the lemma this set contains a countable dense set, hence all the space.
\hfill$\Box$

It follows that if the bottom structure satisfies (BC) $\Phi$ has a density.
\subsection{Application to SDE's.}
Let $d\in\Ne$, we consider the following SDE :
\begin{equation}\label{eq}
X_t =x+\int_0^t \int_X c(s,X_{s^-},u)\tN (ds,du)+\int_0^t \sigma
(s,X_{s^-})dZ_s
\end{equation}
where  $x\in\R^d$,  $c:\R^+
\times \R^d \times X\rightarrow\R^d$ and $\sigma:\R^+ \times \R^d \rightarrow\R^{d\times n}$, $Z$ is a semi-martingale and $\tilde{N}$ a compensated Poisson measure.

The lent particle method allows to apply the machinery of Malliavin calculus faster than usual and under a set of hypotheses that express the Lipschitz character of the coefficient and some other regularity assumptions for the details of which we refer to \cite{bouleau-denis2}.

Let us emphasize that applying the method to SDE's uses reasoning in complete functional spaces in which may be computed and solved the stochastic differential equations giving the $\sharp$ of the solution. This takes full advantage of the fact that the lent particle formula is proved not only on a set of test functions but on the space $\mathbb{D}$ itself.

In \cite{bouleau-denis2} applications are given to McKean-Vlasov type equation driven by a L\'evy process and to stable like processes.
\subsubsection{Example 9. A regular case violating H\"ormander conditions.} The following SDE driven by a two dimensional Brownian motion
\begin{equation}\label{diffusion}\left\{\begin{array}{rl}
X^1_t&=z_1+\int_0^tdB^1_s\\
X^2_t&=z_2+\int_0^t2X^1_sdB^1_s+\int_0^tdB^2_s\\
X^3_t&=z_3+\int_0^tX^1_sdB^1_s+2\int_0^tdB^2_s.
\end{array}\right.
\end{equation}
 is degenerate and the H\"ormander conditions are not fulfilled. The generator is $A=\frac{1}{2}(U_1^2+U_2^2)+V$
and its adjoint $A^\ast=\frac{1}{2}(U_1^2+U_2^2)-V$ with  
$U_1=\frac{\partial}{\partial x_1}+2x_1\frac{\partial}{\partial x_2}+x_1\frac{\partial}{\partial x_3}$,
$U_2=\frac{\partial}{\partial x_2}+2\frac{\partial}{\partial x_3}$ and $V=-\frac{\partial}{\partial z_2}-\frac{1}{2}\frac{\partial}{\partial z_3}$. The Lie brackets of these vectors vanish and the Lie algebra is of dimension 2:
the diffusion remains on the quadric of equation
$\frac{3}{4}x_1^2-x_2+\frac{1}{2}x_3-\frac{3}{4}t=C.$

Let us now consider the same equation driven by a L\'evy process :

\begin{equation}\label{machin}\left\{\begin{array}{rl}
Z^1_t&=z_1+\int_0^tdY^1_s\\
Z^2_t&=z_2+\int_0^t2Z^1_{s_-}dY^1_s+\int_0^tdY^2_s\\
Z^3_t&=z_3+\int_0^tZ^1_{s_-}dY^1_s+2\int_0^tdY^2_s
\end{array}\right.
\end{equation}
under hypotheses on the L\'evy measure  such that the bottom space may be equipped with the carr\'e du champ operator 
$\gamma[f]=y_1^2f^{\prime 2}_1+y_2^2f^{\prime 2}_2$ satisfying (BC) and (EID).
Applying the lent particle method is as usual and shows easily that if the L\'evy measures of  $Y^1$ and $Y^2$ are infinite $Z_t$ has a density on $\mathbb{R}^3$. See \cite{bouleau-denis} for details. The regularizing property is related to the fact that equation (\ref{machin}) is not under the canonical form in the sense of Kunita \cite{kunita} \cite{kunita2}. The next example, on the contrary shows a L\'evy process in $\mathbb{R}^3$ living on a hyperbolic paraboloid.

\subsubsection{Example 10.} For $\alpha\in\mathbb{R}^3$, let us consider the diffusion solution of
$$X_t=\alpha+\int_0^tU_1(X_s)\circ dB^1_s+\int_0^tU_2(X_s)\circ dB^2_s$$ where $B=(B^1,B^2)$ is a standard Brownian motion with values in $\mathbb{R}^2$, integrals being in the Stratonovich sense, and vectors $U_1$ and $U_2$ being given by
$$U_1(x)=\left(\begin{array}{c}
x_1x_3^2-a_0x_2x_3\\
x_2x_3^2+a_0x_1x_3\\
x_3(a_0^2+x_3^2)
\end{array}\right)
\qquad
U_2(x)=\left(\begin{array}{c}
x_2x_3^2+a_0x_1x_3\\
x_1x_3^2-a_0x_2x_3\\
x_3(a_0^2+x_3^2)
\end{array}\right)
$$ with $a_0=\alpha_1^2+\alpha_2^2-\alpha_3^2$.

Then the diffusion $(Z_t)$ remains on the quadric of equation
\begin{equation}\label{PH}x_1^2+x_2^2-x_3^2=a_0^2.\end{equation}
Now let us consider two independent L\'evy processes $(Y^1_t),(Y^2_t)$ and the equation
\begin{equation}\label{Z}Z_t=\alpha+\int_0^t U_1(Z_{s-})dY^1_s+\int_0^t U_2(Z_{s-})dY^2_s\end{equation} the Markov process with jumps $Z$ remains on the hyperbolic paraboloid (\ref{PH}) as seen by applying Ito formula. This is due to the fact that the HP   is a ruled manifold and at each point of it the jumps of $Z$ are in the direction of either  generatrix crossing at this point. Equation (\ref{Z}) is canonical in Kunita's sense. Using a map from the HP to $\mathbb{R}^2$ the method allows to show the density of the law of $Z_t$ w.r. to the area measure on the HP.
\subsection{A useful theorem of Paul L\'evy.}
It is the occasion to rectify a historical injustice about the remarkable article of Paul L\'evy 
``Sur les s\'eries dont les termes sont des variables \'eventuelles ind\'ependantes" which appeared in {\it Studia Mathematica} in 1931 \cite{levy}. This article is almost never cited up to now (today the search engins do not mention any citation of this article) and the textbooks of K.I. Sato \cite{sato2} and of J. Bertoin \cite{bertoin} do not quote it.
One of his theorems, that we recall below,  is generally attributed to Hartman and Wintner ``On the infinitesimal generator of integral convolutions" {\it Amer. J. Math.} 64, (1942) 273-298, which was published ten years later.

Paul L\'evy's results\footnote{There is an obvious misprint in this paper p128 line 20 where = has to be change into $\neq$.}
 may be stated as follows:
\begin{Th}{\it Let $X_n$ be a sequence of independent real random variables such that the series $\sum X_n$ converges almost surely.

a) If for any sequence of constants $(a_n)$, $\sum\mathbb{P}\{X_n \neq a_n\}$ diverges, $\sum X_n$ has a continuous law.

b) If there is a sequence $(a_n)$ s.t. $\sum\mathbb{P}\{X_n \neq a_n\}$ converges and if the lower bound of the total mass of the discrete part of the laws of the $X_n$'s is zero, then the law of $\sum X_n$ is continuous.}
\end{Th}
It follows from this theorem that any process with independent increments whose L\'evy measure in infinite has a continuous law. In the framework of random Poisson measures it gives easily Lemma \ref{lemme de levy} above.

\begin{Rq} {\rm If $f\in L^1(\nu)$ and $\nu\{f\neq 0\}=+\infty$ then the law of $N(f)$ is continuous but its characteristic function does not necessarily tend to zero at infinity, in other words is not necessarily a Rajchman measure (cf \cite{rajchman1} \cite{rajchman2} or \cite{bouleauhal}) . This gives an easy way to construct continuous measures which are not Rajchman.
Let $m=f_*\nu$, $m$ is $\sigma$-finite and integrates $x\mapsto|x|$. Since $\mathbb{E}e^{iuNf}=e^{\int(e^{iuf}-1+iuf)d\nu}$ the law of 
$Nf$ is Rajchman iff $\lim_{|u|\rightarrow+\infty}\int(1-\cos ux)m(dx)=+\infty$. If we choose a step function for $f$ so that $m=\sum \varepsilon_{\frac{1}{2^n}}$, we have $\int(1-\cos 2^k\pi x)m(dx)=\sum_{j=0}^\infty(1-\cos\frac{\pi}{2^j})<+\infty$ so that the law of $Nf$ is continuous and not Rajchman.}\end{Rq}

\section{Regularity results for multiple Poisson integrals.}
Let us first recall some links of our study with the Fock space.
\subsection{Random Poisson measure and Fock space.}
We recall that $\nu$ is continuous (i.e.
diffuse).
Let us call {\it simple} the measurable functions $f$ defined on $(X^m ,\cX^{\otimes m})$ which are symmetric, finite sums of weighted indicator functions of sets of the form $A_1\times\cdots\times A_m$ with disjoint $A_i$'s. 

On simple functions if we define $$I_m(f)=\int_{X^m}f(x_1,\ldots,x_m)\tilde{N}(dx_1)\cdots\tilde{N}(dx_m)$$ it is easily seen that
$$\mathbb{E}[I_m(f)I_n(g)]=\delta_{m,n}n!\langle f,g\rangle_{ L^2_{} (X^m ,\cX^{\otimes m}, \nu^{\times m})}.$$
Thanks to this equality $I_m(f)$ may be extended to  $f\in L^2_{} (X^m ,\cX^{\otimes m}, \nu^{\times m})$ so that denoting $\tilde{f}$ the symmetrized $f$, $I_m(f)=I_m(\tilde{f})$ and
$$ \mathbb{E}[I_m(f)I_n(g)]=\delta_{m,n}n!\langle \tilde{f},\tilde{g}\rangle_{ L^2_{} (X^m ,\cX^{\otimes m}, \nu^{\times m})}.$$
Let us observe that for  $f\in L^2_{} (X^m ,\cX^{\otimes m}, \nu^{\times m})$ the formula
 \[ I_m (f)=\int_{X^m} f(x_1 ,\cdots ,x_m ){\bf 1}_{\{\forall i\neq j ,
x_i \neq x_j\}}\, \tN (dx_1 )\cdots \tN (dx_m ).\] 
is a symbolic notation, because on the right hand side, the quantities to be substracted to the integral on $X^m$ are generally not defined for non regular functions $f$. 

It has a sense if $f$  is well defined on diagonals by continuity, $X$ being supposed topological. 
A sense may also be  yielded by Hilbertian methods, supposing $f$ allows to define trace operators.

  The sub-vector space of $L^2 (\Om,\cA, \bbP)$ generated by the
variables $I_n (f)$, $f\in L^2_{} (X^n ,\cX^{\otimes n},
\nu^{\times n})$ is the Poisson chaos of order $n$ denoted $C_n$. The equality
\begin{equation}L^2 (\Om,\cA, \bbP)=\R\oplus_{n=1}^{+\infty} C_n .\end{equation}
has been proved by K. Ito (see \cite{ito}) in 1956. This proof  is based on
the fact that the set $\{ N(E_1)\cdots N(E_k),$ $ (E_i)$ disjoint sets in  $\cX \}$ is total in $L^2 (\Om,\cA, \bbP)$.

There are now several proofs of this result. A combinatorial proof is possible by counting the role of successive diagonals (cf \cite{rota-wallstrom} and \cite{bouleau-denis} \S4.1.) By transportation of structure, the density of the chaos has a short proof using stochastic calculus for the Poisson process on $\mathbb{R}_+$ (cf Dellacherie- Maisonneuve-Meyer \cite{dellacherie} p207).

Thanks to the density of the chaos the following expansion is easily obtained (cf \cite{surgailis}) for $u\in L^1\cap L^\infty(\nu)$ with small $\|u\|_\infty$,
\begin{equation}\label{252}
e^{N(\log(1+u))-\nu(u)}=1+\sum_{n=1}^{+\infty}\frac{1}{n!}I_n (u^{\otimes
n}).\end{equation} 
Let us mention the relationship between the strongly continuous semigroup of the bottom structure $p_t$ in $L^2(\nu)$ and the one of the upper structure $P_t$ in $L^2(\mathbb{P})$ (see \cite{bouleau-denis} for a proof).
For all $u$  measurable function with
$-\frac12 \leq u\leq 0$,
\begin{equation}\label{Pt} \forall t\geq 0 ,\ P_t [e^{N(\log (1+u))}]=e^{N(\log(1+p_t
u))}.\end{equation}
By (\ref{252}) and (\ref{Pt}) the vector spaces $C_n$ are preserved by $P_t$ and 
\begin{equation}\label{253}
P_t(I_n(u^{\otimes n}))=I_n((p_tu)^{\otimes n})).
\end{equation} It is generally spoken of second quantization for the transform $(p_t)\mapsto(P_t)$. More precisely the second quantization maps the generator $a$ of $p_t$ to an operator on the Fock space which may be then lifted up either on the Wiener space or on the Poisson space and in this later case corresponds to the generator $A$ of $P_t$.
\begin{Rq}{\rm 
Let us suppose that the bottom semigroup $p_t$ be generated by a transition kernel $\tilde{p_t}(x,dy)$ from $(X,\mathcal{X})$ into itself, which be simulatable in the sense that there exists a probability space -- that we choose here for the sake of simplicity of notation to be $(R,\mathcal{R},\rho)$ -- and a family of random variables $\eta_t(x,r)$ such that the law of $\eta_t(x,r)$ under $\rho(dr)$ be $\tilde{p_t}(x,dy)$ .

Then, using our notation in which we have  $\omega=\int\varepsilon_x\;N(dx)$, the fact that the upper semigroup represents the evolution of independent particles each governed by $p_t$ and with initial law $N$ (see the introduction of \cite{bouleau-denis}) may be expressed, for $F$ $\mathcal{A}$-measurable and bounded, by the formula
\begin{equation}\label{Pt2}
P_tF=\hat{\mathbb{E}}F(\int\varepsilon_{\eta_t(x,r)}\;N\odot\rho(dxdr))
\end{equation} in analogy with the Mehler formula for the Ornstein-Uhlenbeck semigroup on the Wiener space or extensions of it (see \cite{bouleau3} p116). Applying (\ref{Pt2}) to $F=\exp{N\log(1+g)}$ for $-\frac{1}{2}\leq g\leq 0$ gives
$$P_tF=\hat{\mathbb{E}}\exp\int\log(1+g(\eta_t(x,r))\;N\odot\rho(dxdr)$$ what by formula (\ref{fm}) leads anew to (\ref{Pt}) by a different way :

\hspace{2.5cm}$P_tF=\exp N\log(\int(1+g(\eta_t(x,r))\rho(dr))=\exp N\log(1+p_tg).$ \hfill$\Box$
}\end{Rq}
\begin{Rq} {\rm Surgailis \cite{surgailis} has shown that in the correspondence between $p_t$ and $P_t$ given by (\ref{253}) a necessary and sufficient condition $P_t$ be Markov is that $p_t$ and its adjoint be Markov operators (i.e. positivity preserving and s.t. $p_t1\leq1$).

In our framework $p_t$ is selfadjoint and so is $P_t$.\hfill$\Box$
}\end{Rq}
\subsection{Decomposition of $\mathbb{D}$ in chaos.}
Let us precise some notation. On the upper space $(\Omega, \mathcal{A}, \mathbb{P}, \mathbb{D}, \Gamma)$ the Dirichlet form is denoted $\mathcal{E}$.

The product structure $(X,\mathcal{X},\nu,\mathbf{ d},\gamma)^n$ will be denoted  $(X^n,\mathcal{X}^{\otimes n},\nu^{\times n},\mathbf{ d}_n,\gamma_n)$ (cf  \cite{bouleau-hirsch2} Chap V). It is endowed with the Dirichlet form $e_n[f]=\frac{1}{2}\int\gamma_n[f]d\nu$. The functions in $\mathbf{ d}_n$ which are symmetric define a sub-structure of $(X^n,\mathcal{X}^{\otimes n},\nu^{\times n},\mathbf{ d}_n,\gamma_n)$ denoted $(X^n,\mathcal{X}^{\otimes n}_{sym},\nu^{\times n},\mathbf{ d}_{n,sym},\gamma_n)$.  The semigroup associated with $e_n$ is denoted $p^{\otimes n}_t$. Our choice of gradient for the bottom space (see \S2.3 above) induces a gradient for $(X^n,\mathcal{X}^{\otimes n},\nu^{\times n},\mathbf{ d}_n,\gamma_n)$ that we denote $(\cdot)^{\flat_n}$ with values in $(L_0^2(R,\mathcal{R},\rho))^{\otimes n}$ :
$$(f^{\flat_n})(x_1,r_1,x_2,r_2,\cdots,x_n,r_n)=(f(\cdot, x_2,\cdots,x_n))^\flat(x_1,r_1)+
(f(x_1, \cdot,x_3,\cdots,x_n))^\flat(x_2,r_2)+\cdots$$
let us note that if $f$ is symmetric, then $f^{\flat_n}$ is symmetric of the pairs $(x_i,r_i)$.\\

\noindent Let be $f(x_1,\ldots, x_m)=f_1(x_1)\cdots f_m(x_m)\in\mathbf{ d}_m$ and 
$g(x_1,\ldots, x_n)=g_1(x_1)\cdots g_n(x_n)\in\mathbf{ d}_n$.  By polarization of  (\ref{253}) $P_tI_mf=I_mp^{\otimes m}_tf$ gives
$$\begin{array}{rcl}
\mathcal{E}_t[I_mf,I_ng]&=&\frac{1}{t}\langle I_mf-P_tI_mf,I_ng\rangle_{L^2(\mathbb{P})}\;
=\; \frac{1}{t}\langle I_m(f-p^{\otimes m}_tf),I_ng\rangle\\
&&\\
&=&\delta_{mn}m!\langle\frac{f-p^{\otimes m}_tf}{t},g\rangle_{L^2(\nu^{\times m})}.
\end{array}
$$ By the theory of symmetric strongly continuous contraction semigroups, we have $F\in \mathbb{D}$ if and only if $\lim_{t\downarrow 0}\uparrow\mathcal{E}_t[F]<+\infty$ and $\mathcal{E}[F]=\lim_{t\downarrow 0}\mathcal{E}_t[F]$. Taking $f=g$, we obtain that $I_mf\in\mathbb{D}$ and $\mathcal{E}[I_mf]=m!e_m[f]$. Then by density we obtain
\begin{Pro}\label{chaos} {\it For $f\in\mathbf{ d}_m$ the random variable $I_mf\;(=I_m(\tilde{f}))$ belongs to $\mathbb{D}$. The vector spaces $D_m$ generated by $I_mf$ for $f\in\mathbf{ d}_m$, are closed and orthogonal in $\mathbb{D}$. The sum
$$\mathbb{D}=\mathbb{R}\bigoplus_{n\geq 1}D_n$$ is direct in the sense of the Hilbert structure of $\mathbb{D}\;$ $\;(\|\cdot\|^2_\mathbb{D}=\|\cdot\|^2_{L^2}+\mathcal{E}[\cdot])$.

Every function $F$ in $\mathbb{D}$ decomposes uniquely
$$F=\mathbb{E}[F]+\sum_{n\geq 1}I_n(F_n)$$ with $F_n\in\mathbf{ d}_n$.
}\end{Pro}
\begin{proof} It remains only to prove the density of the Dirichlet chaos $D_n$. Let be $F\in\mathbb{D}$ and let $F=\sum I_n(F_n)$ be its $L^2$-chaos expansion. Then 
$$\begin{array}{rcl}
\frac{1}{t}\langle F-P_tF,F\rangle_{L^2(\mathbb{P})}&=&\frac{1}{t}\sum_{n\geq1}\langle I_n(F_n-p^{\otimes n}_tF_n),I_nF_n\rangle_{L^2(\mathbb{P})}\\
&&\\
&=&\sum_{n\geq 1}n!\langle\frac{F_n-p^{\otimes n}_tF_n}{t},F_n\rangle_{L^2(\nu^{\times n})}.
\end{array}
$$ Since on the left-hand side $\frac{1}{t}\langle F-P_tF,F\rangle\uparrow \mathcal{E}[F]<+\infty$ it follows that all terms on the right-hand side, which are increasing, possess  limits what yields $F_n\in\mathbf{ d}_n$ and the proposition follows.
\end{proof}
Let us emphasize that this proof is only based on the relation of second quantization (\ref{253}) and would be still valid on the Wiener space for instance equipped with a generalized Mehler type structure (cf e.g. \cite{bouleau3} p113 et seq.) or on the Poisson space equipped with a non local Dirichlet form on the bottom space.\\

Let $u\in L^\infty\cap\mathbf{ d}$, applying the gradient operator $\sharp$ to the two sides of 
(\ref{252}) gives
$$e^{N\log(1+tu)-t\nu(u)}\int\frac{tu^\flat}{1+tu}\,dN\odot\rho=
\sum_{n\geq 1}\frac{t^n}{n!}(I_n(u^{\otimes n}))^\sharp$$ what yields, taking terms in $t^n$ on both sides
\begin{equation}\label{diese1}(I_n(u^{\otimes n}))^\sharp= \sum_{q=0}^{n-1}(-1)^q\frac{n!}{(n-1-q)!}I_{n-1-q}(u^{\otimes(n-1-q)})\int u^qu^\flat\;dN\odot\rho.\end{equation} 
and 
\begin{equation}\label{GamGam}\frac{1}{i!}\frac{1}{j!}\Gamma[I_iu^{\otimes i},I_jv^{\otimes j}]=\int\left(\sum_{k=1}^i\frac{I_{i-k}u^{\otimes(i-k)}}{(i-k)!}(-1)^ku^{k-1}\right)\left(\sum_{\ell=1}^j\frac{I_{j-\ell}v^{\otimes(j-\ell)}}{(j-\ell)!}(-1)^\ell v^{\ell-1}\right)\gamma[u,v]\,dN.\end{equation}
If $f$ is the symmetrized of $f_1(x_1)\cdots f_m(x_m)$ then (\ref{diese1}) writes
\begin{equation}\label{260}
I_m(f)^\sharp=\int \left(mI_{m-1}f^\flat-m(m-1)I_{m-2}f^\flat+m(m-1)(m-2)I_{m-3}f^\flat-\cdots\right)\;dN\odot\rho
\end{equation} where $I_{m-p}$ acts on the $m-p$ first arguments of $f$ and $\flat$ acts on the last one, all free arguments being taken on the same point $x$.

Extending formulae (\ref{diese1})-(\ref{260}) from tensor products to general functions $f\in\mathbf{ d}_m$ supposes {\it a priori}  that $f$ does possess {\it traces} on diagonals. Indeed let us suppose $f$ and $g$ be regular so that values on diagonals make sense, then defining  for regular symmetric functions $f(x_1,\ldots,x_m)$ and $g(y_1,\ldots,y_n)$ the $(k,\ell)$-$\gamma$-contraction, for $1\leq k\leq m$ and $1\leq \ell\leq n$, denoted $f\!{\tiny\begin{array}{c}\gamma\\
\asymp\\
{k,\ell}
\end{array}}\!g$ as follows
$$\begin{array}{rl}
f\!{\footnotesize\begin{array}{c}\gamma\\
\asymp\\
{k,\ell}
\end{array}}\!g&\!\!\!(x_1,\cdots,x_{m-k},y_1,\cdots,y_{n-\ell},x)=\\
&\gamma[f(x_1,\cdots,x_{m-k},x,\cdots,x,\cdot),g(y_1,\cdots,y_{n-\ell},x,\cdots,x,\cdot)](x),
\end{array}$$ the function $f\!{\tiny\begin{array}{c}\gamma\\
\asymp\\
{k,\ell}
\end{array}}\!g$ is symmetric in $(x_1,\cdots,x_{m-k})$ and in $(y_1,\cdots,y_{n-\ell})$.  Then formulae (\ref{diese1})-(\ref{260}) extend to symmetric functions $f$ and $g$ as
\begin{equation}\label{gamma}
\Gamma[I_m(f),I_n(g)]=\sum_{k=1}^m\sum_{\ell=1}^n(-1)^{k+\ell}\frac{m!n!}{(m-k)!n-\ell)!}
I_{m-k}I_{n-\ell}\int(f\!{\footnotesize\begin{array}{c}\gamma\\
\asymp\\
{k,\ell}
\end{array}}\!g)\;dN.
\end{equation} where $I_{m-k}$ operates on the $x_i$'s, $I_{n-\ell}$ operates on the $y_j$'s and $N$ on $x$.

But this formula is unsatisfactory because we know that $I_m(f)$ is defined and in $\mathbb{D}$ for general functions $f\in\mathbf{ d}_m$ which dont have defined values on diagonals in general. Actually the values on diagonals cancel in formula (\ref{gamma}). To see this we have to consider the Fock space for the gradient and to come back to the lent particle formula.

The random Poisson measure $N\odot\rho$ (cf \S2.2) is defined on $(\Omega\times\hat{\Omega},\mathcal{A}\otimes\hat{\mathcal{A}},\mathbb{P}\times\hat{\mathbb{P}})$ with intensity $\nu\times\rho$ on $(X\times R,\mathcal{X}\otimes\mathcal{R})$. It possesses an expansion in chaos : $\forall F\in L^2(\mathbb{P}\times\hat{\mathbb{P}})$
$$F=\mathbb{E}\hat{\mathbb{E}}F+\sum_{n\geq 1}J_n(F_n)$$ where $J_n$ denotes the multiple integral for $\widetilde{N\odot\rho}$ and where $F_n\in L^2_{sym}((\nu\times\rho)^{\times n})$.

Let us remark that the random Poisson measure $N$ may be seen as a function of $N\odot\rho$ and that the multiple integrals $I_n$ are nothing else but $J_n$ applied to a function $G(x_1,r_1,\cdots,x_n,r_n)$ not depending on the $r_i$'s. We can now state
\begin{Pro}\label{pro-diese}{\it  Let be $f\in\mathbf{ d}_{m,sym}$, by {\rm Prop \ref{chaos}} the multiple integral $I_m(f)$ belongs to $\mathbb{D}$. 

a) Its gradient is given by 
\begin{equation}\label{diese5}
(I_m(f))^\sharp=\int \anni(I_{m-1}(f))^\flat dN\odot\rho=m\int I_{m-1}(\varphi)\;N\odot\rho(dxdr)
\end{equation}
where we note $\psi(x_1,\ldots,x_{m-1},x,r)=(f(x_1,\ldots,x_{m-1},\cdot))^\flat(x,r)$ and $\varphi$ is defined as
$$\varphi(x_1,\ldots,x_{m-1},x,r)=\psi(x_1,\ldots,x_{m-1},x,r)1_{\{x_i\neq x\;\forall i=1,\ldots,m-1\}}$$ so that $\varphi(\cdot,\cdots,\cdot,x,r)\in L^2_{sym}(\nu^{\times(m-1)})$ and $I_{m-1}(\varphi)$ is defined. 

b) This gradient may also be written
\begin{equation}\label{gradient-chaos} (I_mf)^\sharp=J_m(f^{\flat_m})
\end{equation} so that
\begin{equation}\label{Gamma5}
\Gamma[I_m(f),I_n(g)]=\hat{\mathbb{E}}[J_m(f^{\flat_m})J_n(g^{\flat_n})].
\end{equation}
}\end{Pro}
\begin{proof} Let be $f\in\mathbf{ d}_{m,sym}$. Let us apply the lent particle formula to $I_m(f)$. We have 
$$\crea I_m(f)=I_m(f)+mI_{m-1}f\qquad\mathbb{P}\times\nu\mbox{-a.e.}$$
and since $I_m(f)$ does not depend on $x$ $$(\crea I_m(f))^\flat=m(I_{m-1}f)^\flat\qquad\mathbb{P}\times\nu\times\rho\mbox{-a.e.}$$
Now, applying the operator $\anni$ amounts to take the preceding relation with $\omega$ changed into $\anni\omega$ and to work under the measure $\mathbb{P}_N$ instead of $\mathbb{P}\times\nu$. That means  that a functional $F(\tilde{N}(u),x)$ is changed into
$$\anni_x(F(\tilde{N}(u),x))=F(\int u(y)1_{\{y\neq x\}}\tilde{N}(dy),x)\quad \mathbb{P}_N\mbox{-a.e.}$$ Taking $m=2$ for instance, we see that $\anni(I_1f)$ must be written $\mathbb{P}_N\mbox{-a.e.}$ 
$\int f(y,x)1_{\{y\neq x\}}\tilde{N}(dy)$ instead of $(I_1f)(x)-f(x,x)$. Thus the part a) of the statement is a direct application of the lent particle formula.

b) Since $\flat$ takes its values in $L_0^2(R,\mathcal{R},\rho)$, it is equivalent to use the compensated random measure $\widetilde{N\odot\rho}$ instead of $N\odot\rho$ in (\ref{diese5}).

Now $m\int I_{m-1}(\varphi)d\widetilde{N\odot\rho}=J_m(f^{\flat_m})$ as seen by beginning with $f=u^{\otimes m}$, then polarizing to $f$ symmetrized of $u_1\otimes\cdots\otimes u_m$ and then to general $f\in\mathbf{ d}_{n,sym}$ by density.
\end{proof}
Let us remark that formula (\ref{diese5})  allows a new simple proof of the orthogonality of the chaos in $\mathbb{D}$. Let $f$ be as in the proposition. We have
$$
\begin{array}{rl}
2\mathcal{E}[I_mf]=\mathbb{E}\Gamma[I_mf]&=\mathbb{E}m^2\int\gamma[I_{m-1}(f1_{\{x_i\neq x\forall i\}})]\;N(dx)\\
&= m^2\int\anni\gamma[I_{m-1}f]\;dNd\mathbb{P}\\
&=m^2\int\gamma[I_{m-1}f]d\mathbb{P}d\nu\qquad\qquad\mbox{(by Lemma \ref{lem8})}\\
&=m^2(m-1)!\int\gamma[f]d\nu^{\times(m-1)}d\nu\\
&=m!2e_m[f].
\end{array}
$$ and similarly with the scalar products. Now (\ref{gradient-chaos}) yields an even shorter proof using the orthogonality of the chaos generated by $J_n$ under $\mathbb{P}\times\hat{\mathbb{P}}$, since $\langle f^{\flat_m},g^{\flat_m}\rangle_{L^2(\nu\times\rho)^m}=2e_m[f,g]$.\\

\noindent Contrarily to the Wiener case the random variables $I_m(f)$ are not regular in general. Their distributions may contain Dirac masses. Even in the first chaos the $\sharp$ or the $\Gamma$ applied to $I_1u=\tilde{N}u$ yields a non deterministic result, and the sharp operator does not diminish the order of the chaos. Studying regularity of multiple integrals needs therefore additional hypotheses.

\subsection{Density for $(I_1(g),\ldots,I_n(g^{\otimes n}))$.}
Relation (\ref{diese1}) yields immediately
\begin{equation}\label{Gam}\frac{1}{i!}\frac{1}{j!}\Gamma[I_i,I_j]=\int\left(\sum_{k=1}^i\frac{I_{i-k}}{(i-k)!}(-1)^kg^{k-1}\right)\left(\sum_{\ell=1}^j\frac{I_{j-\ell}}{(j-\ell)!}(-1)^\ell g^{\ell-1}\right)\gamma[g]\,dN.\end{equation}
Let us denote $\mathcal{I}$ the column vector of $(I_1,\ldots,I_n)$, we have
$$\Gamma[\mathcal{I},\mathcal{I}^t]=\int V V^t\gamma[g]\,dN$$
with $V$ the column vector of
$(-1,-I_1+g,\ldots, n!\sum_{k=1}^n\frac{I_{n-1}}{(n-k)!}(-1)^kg^{k-1})$.

\noindent Let us precise now some hypotheses. We suppose $\nu\{\gamma[g]>0\}=+\infty$ and that assumptions are fulfilled such that we have (BC) on the bottom space and (EID) on the upper
space, as usual.

If for some $\omega\in\Omega$ the matrix $\Gamma[\mathcal{I},\mathcal{I}^t]$ is singular, this means that all the vectors 
$$V(\omega,X_i(\omega))\quad\mbox{ for } X_i\in\mbox{supp}(\omega)\cap\{\gamma[g]>0\}$$
belong to the same hyperplan of $\mathbb{R}^n$, in other words, this implies that there exist $\lambda_0(\omega),\ldots,\lambda_{n-1}(\omega)$ not all null such that:
$$-\lambda_0(\omega)+\lambda_1(\omega)(-I_1+g)+\cdots+\lambda_{n-1}(\omega)n!\sum_{k=1}^n\frac{I_{n-1}}{(n-k)!}(-1)^kg^{k-1}=0$$
on all the points of supp$( \omega)\cap\{\gamma[g]>0\}.$

Since $g\in\bbd$, by (EID) on the bottom space --- which is always true for scalar functions --- the measure $g_*[1_{\{\gamma[g]>0\}}.\nu]$ is absolutely continuous hence continuous (diffuse). As $\nu\{\gamma[g]>0\}=+\infty$ the random Poisson measure image by $g$ of the points of $N$ which are in $\{\gamma[g]>0\}$ do possess infinitely many distinct points. Hence the $g(X_i(\omega))$ cannot annul a polynomial except if it is identically sero.

The question is therefore to know whether 
$$-\lambda_0(\omega)+\lambda_1(\omega)(-I_1+x)+\cdots+\lambda_{n-1}(\omega)n!\sum_{k=1}^n\frac{I_{n-1}}{(n-k)!}(-1)^kx^{k-1}\equiv 0$$
implies $\lambda_0(\omega)=\cdots=\lambda_{n-1}(\omega)=0$.

But this is due to the fact that the annulation of the coefficients of this polynomial builds a triangular linear system whose diagonal terms are $-\lambda_0(\omega),\ldots,n!(-1)^{n}\lambda_{n-1}(\omega)$. We have proved
\begin{Pro}{\it If the upper structure satisfies {\rm(EID)}, for $g\in  L^\infty\cap\bbd$ such that \break $\nu\{\gamma[g]>0\}=+\infty$ the vector
$(I_1(g),\ldots,I_n(g^{\otimes n}))$
has a density on $\mathbb{R}^n$.}\end{Pro}
\begin{Rq}{\rm This result is quite different from what happens on the Wiener space since there the law of $(I_1(f),\ldots,I_n(f^{\otimes n}))$ is carried by the algebraic curve of equation
$$
\left\{
\begin{array}{c}
x_2=2!H_2(\|f\|^2,x_1)\\
\vdots\\
x_n=n!H_n(\|f\|^2,x_1)
\end{array}
\right.
$$
where $H_n(\lambda,x)$ is the Hermite polynomial given by 
$$\exp(tx-\frac{t^2\lambda}{2})=\sum_{n=0}^\infty t^nH_n(\lambda,x).$$ 
}\end{Rq}
\subsection{Density for $(I_{n_1}(f_1^{\otimes n_1}),\ldots,I_{n_p}(f_p^{\otimes n_p}))$.}
Let $f=(f_1,\ldots,f_p)\in(L^1\cap L^\infty\cap\bbd)^p$ and let be $\mathcal{J}$ the column vector $(I_{n_1}(f_1^{\otimes n_1}),\ldots,I_{n_p}(f_p^{\otimes n_p}))$ where we suppose $n_i\geq 1\;\forall i.$

Defining the polynomials $\mathcal{P}_{i}$ by $\mathcal{P}_{i}(x)=i!\left(\sum_{k=1}^i\frac{I_{i-k}}{(i-k)!}(-1)^kx^{k-1}\right)$ we have by (\ref{Gam}) the equality between $p\times p$-matrices
$$\Gamma[\mathcal{J},\mathcal{J}^t]=
\int\left(\mathcal{P}_{n_i}(f_i)\mathcal{P}_{n_j}(f_j)\gamma[f_i,f_j]\right)_{ij}\;dN.$$
By Lemma \ref{determin}
$$\{\det \Gamma[\mathcal{J},\mathcal{J}^t]=0\}\subset\{\int\det\gamma[f,f^t]
\mathcal{P}_{n_1}(f_1)^2\cdots\mathcal{P}_{n_p}(f_p)^2\;dN=0\}.$$
Let us assume $\nu\{\det\gamma[f,f^t]>0\}=+\infty$, and that we have (EID) below and above. The image by $f$ of $1_{\{\det[f,f^t]>0\}}\cdot\nu$ is absolutely continuous w.r. to Lebesgue measure and the Poisson random measure image of $N\mid_{\{\det[f,f^t]>0\}}$ has an absolutely continuous and infinite  intensity measure, it possesses necessarily points outside the finite union (less than $\sum_{i=1}^p(n_i-1)$) of hyperplans defined by 
$\mathcal{P}_{n_i}(x_i)=0$ whose term of highest degree is $(n_i)!(-1)^{n_i}x_i^{n_i-1}$. We obtain
\begin{Pro}{\it If {\rm(EID)} holds below and above, and if $\gamma[f,f^t]$ is invertible $\nu$-a.e. \break $(I_{n_1}(f_1^{\otimes n_1}),\ldots,I_{n_p}(f_p^{\otimes n_p}))$ has a density as soon as 
$n_i\geq 1\;\forall i.$
}\end{Pro}
\begin{Rq}{\rm Let us compare with the situation on the Wiener space. We dispose only of sufficient conditions of regularity, but we can nevertheless compare the thread of the arguments.

We have $DI_n(g^{\otimes n})=nI_{n-1}(g^{\otimes(n-1)})g$ and 
$$\Gamma[I_{n_i}(f_i^{\otimes n_i}),I_{n_j}(f_{n_j}^{\otimes n_j})]=
n_in_jI_{n_i-1}(f_i^{\otimes (n_i-1)})I_{n_j-1}(f_{n_j-1}^{\otimes (n_j-1)})\int f_if_j\,dt.$$
Since (EID) holds on the Wiener space a sufficient condition of density of $\mathcal{J}$ is that almost surely the vector
$$\left(n_1I_{n_1-1}(f_1^{\otimes (n_1-1)})f_1(t),\ldots, n_pI_{n_p-1}(f_p^{\otimes (n_p-1)})f_p(t)\right)$$
generates a $p$-dimensional space when $t$ varies. It is easily seen by induction on $n$ that\break $\forall f\in L^2(dt), \|f\|\neq0: \mathbb{P}\{I_n(f^{\otimes n})=0\}=0$. It follows that on the Wiener space, $\mathcal{J}$ has a density as soon as $n_i\geq 1\;\forall i$ and $(f_1,\ldots,f_p)$ are linearily independent in $L^2(dt)$.
}\end{Rq}
\subsection{Other functionals of Poisson integrals.}
\subsubsection{Density of $(N(f_1(g),\ldots,N(f_n(g)))$.}
Let $g\in L^\infty\cap\bbd$ and let $f_i$ be regular real functions on $\mathbb{R}$. Let us denote $\mathcal{K}=(N(f_1(g),\ldots,N(f_n(g)))^t$ and suppose  $\nu\{\gamma[g]>0\}=+\infty$. From $\Gamma[N(f_i(g)),N(f_j(g))]=\int f^\prime_i(g)f^\prime_j(g)\gamma[g]dN$ we obtain that the matrix $\Gamma[\mathcal{K},\mathcal{K}^t]$ is singular if the vectors $(f^\prime_1(g),\ldots,f^\prime_n(g))$ taken on the points of $\omega$ are in a same hyperplan. Now the points $g(x),\;x\in\mbox{supp}(\omega),$ have an accumulation point at zero. We obtain
\begin{Pro}{\it Suppose {\rm(EID)} holds above, $g\in L^\infty\cap\bbd$,  $\nu\{\gamma[g]>0\}=+\infty$, and the functions $f_i$ be analytic at the neighborhood of O such that $(1,f_1,\ldots,f_n)$ be linearily independent, then $(N(f_1(g)),\ldots,N(f_n(g)))$ has a density.}\end{Pro}
Since there are infinitely many distinct points $g(x),\;x\in\mbox{supp}(\omega),$ we see also that without analyticity hypothesis it suffices that any hyperplan cuts the curve $(f_1^\prime(t),\ldots,f_m^\prime(t))_{t\in\mathbb{R}}$ at a finite number of points, the $f_i$ being supposed $\mathcal{C}^1\cap Lip$.
\subsubsection{Density of $(\sum_jN(f_j),\ldots,\sum_j(N(f_j))^n).$} Let us consider $\Phi$ the column vector of the polynomials \break$\Phi_k(x_1,\ldots,x_n)=\sum_{j=1}^nx_j^k$, $f=(f_1,\ldots,f_n)\in\bbd^n$ and let us pose $V$ the column vector of the $\Phi_k(N(f_1),\ldots,N(f_n))$.
We obtain 
$$
\begin{array}{rcl}
\Gamma[V,V^t]&=&\nabla\Phi(Nf_1,\ldots,Nf_n)\Gamma[Nf,Nf^t](\nabla\Phi)^t(Nf_1,\ldots,Nf_n)\\
&=& \nabla\Phi(Nf_1,\ldots,Nf_n)\int\gamma[f,f^t]dN(\nabla\Phi)^t(Nf_1,\ldots,Nf_n)\\
\det\Gamma[V,V^t]&=&(\det \nabla\Phi(Nf_1,\ldots,Nf_n))^2\det\int\gamma[f,f^t]dN
\end{array}
$$
where $\nabla\Phi$ is the Jacobian matrix of $\Phi$. 

$\det \nabla\Phi$ is a Vandermonde determinant, if $\nu\{f_i\neq f_j\}=+\infty$, $\det \nabla\Phi(Nf_1,\ldots,Nf_n)$ cannot vanish by Paul L\'evy's theorem.

$\int\gamma[f,f^t]dN$ is an infinite sum of non negative symmetric matrices, as before we can state
\begin{Pro} {\it Supposing {\rm(EID)} above, $\nu\{f_i\neq f_j\}=+\infty\;\forall i\neq j$, and $\nu\{\det\gamma[f,f^t]>0\}=+\infty$, then $V$ has a density.}\end{Pro}

\subsection{Density of $I_n(f)$ for $f\in\mathbf{ d}_{n,sym}$.}
If $f(x_1,\cdots,x_n)$ is a symmetric element of $\mathbf{ d}_n$, the function $(f(x_1,\cdots,x_{n-1},\cdot))^\flat(x,r)$ may be seen as a symmetric Hilbert valued function in $\mathbf{ d}_{n-1}(H)$ with $H=L^2(\nu\times\rho)$. So that we can iterate the operator $\flat$ going down on the arguments $(f(x_1,\cdots,x_{n-2},\cdot,\cdot))^{\flat\flat}\in\mathbf{ d}_{n-2}(H\otimes H)$.

Let us apply this with Prop \ref{pro-diese}:
$$\Gamma[I_n(f)]=n^2\int\anni\gamma[I_{n-1}(f)]dN.$$ By Lemma \ref{lem8} for $\Gamma[I_n(f)]$ to be $>0$ it suffices
$$\gamma[I_{n-1}(f)]>0\quad \mathbb{P}\times\nu\mbox{-a.e.}$$
i.e. that $I_{n-1}f^\flat$ be $\neq 0$ $\mathbb{P}\times\nu\times\rho$-a.e. hence it suffices that $\nu\times\rho$-a.e. $I_{n-1}f^\flat$ have a continuous law.

Now getting down the induction and using Paul L\'evy's theorem yields that it suffices that
$$(\nu\times\rho)^{n-1}\mbox{-a.e.}\quad\nu\{x_1:f^{(n-1)\flat}(x_1,x_2,r_2,\ldots,x_n,r_n)\neq 0\}=+\infty.$$
Applying this to the classical case where the bottom space is $\mathbb{R}_+$ equipped with the Lebesgue measure and the form $e[f]=\frac{1}{2}\int f^{\prime 2}(t)dt$, where we can choose $f^\flat=f^\prime\cdot\xi$ with $\xi$ reduced Gaussian, we obtain
\begin{Pro}{\it For $n\geq 2$, $I_n(f)$ has a density if the Lebesgue measure of the set $\{x_1:\frac{\partial^{n-1}\tilde{f}}{\partial x_2\cdots\partial x_n}\neq 0\}$ is infinite $dx_2\cdots dx_n$-a.e.}\end{Pro}
This extends to  the classical case on $\mathbb{R}^d$ taking $f^\flat=\frac{\partial f}{\partial x_1}\xi_1+\cdots+\frac{\partial f}{\partial x_n}\xi_n$ with the $\xi_i$ i.i.d. reduced Gaussian. 

\begin{Rq}{\rm
There is a major difference with the case of the Brownian motion about the sum of the series
$$\sum_{n=0}^\infty\frac{t^n}{n!}I_n(f^{\otimes n}).$$
In the case of Wiener space this sum is a function of $\int fdB=I_1(f)$ since it is equal to $e^{t\int fdB-\frac{1}{2}t^2\|f\|^2}$. On the Poisson space it is not a function of $I_1(f)=N(f)$ but of $N(\log(1+tf))$ 
and for $f\in L^\infty\cap\bbd$ and small $t$ by our usual argument using Paul L\'evy's theorem
the pair \break$(Nf,N\log(1+tf))$ do have a density if $\nu\{\gamma[f]>0\}=+\infty$.

It is natural to ask about the density of the vector $(N\log(1+t_1f),\ldots,N\log(1+t_nf))$.
For $f\in L^\infty\cap\bbd$, supposing $0<t_1,\ldots,t_n<\|f\|_\infty$, by the method it suffices to have (BC) down, (EID) above, $\nu\{\gamma[f]>0\}=+\infty$ and the $t_i$ to be distinct.
}\end{Rq}
\begin{Rq} {\rm In the Wiener case multiple integrals obey a product formula (cf. Shigekawa \cite{shigekawa} p276) allowing to express explicitely $I_m[f]I_n[g]$ as linear combination of multiple integrals of order less or equal to $m+n$. 

A similar formula exists on the Poisson space slightly more complicated. It may be obtained in the following way. Let $u,v\in L^2\cap L^\infty(\nu)$ with small uniform norm. By the relation
$$e^{N(\log(1+su))-s\nu(u)}e^{N(\log(1+tv))-t\nu(v)}=
e^{N(\log(1+su+tv+stuv))-\nu((su+tv+stuv)}e^{st\nu(uv)}$$
thanks to (\ref{252}) we have
$$(1+\sum_{m=1}^\infty \frac{s^m}{m!}I_m(u^{\otimes m})(1+\sum_{n=1}^\infty \frac{t^n}{n!}I_n(v^{\otimes n})=
(1+\sum_{p=1}^\infty \frac{1}{p!}I_p((su+tv+stuv)^{\otimes p})e^{st\nu(uv)}$$
and the product formula is obtained by identification of the term in $s^mt^n$ of the two sides. Then it may be extended by polarization to $\tilde{f}$ and $\tilde{g}$ for $f=f_1\otimes \cdots\otimes f_m$ and  $g=g_1\otimes \cdots\otimes g_n$ and then for general $f\in L^2(\nu^{\times m})$, $g\in L^2(\nu^{\times n})$ by density. See \cite{kabanov}, \cite{privault2}, \cite{tudor} for different forms of such a formula, also \cite{russo-vallois}, \cite{rota-wallstrom}, and \cite{dellacherie} p261 for a general expression and proof.

If we apply this product formula to $J_m(f^{\flat_m})J_n(g^{\flat_n})$ using $\hat{\mathbb{E}}J_k(h)=I_k(\int h\rho(dr_1)\cdots\rho(dr_k))$ for $h(x_1,r_1,\cdots,x_k,r_k)\in L^2_{sym}(\nu\times\rho)^{\times k})$ we could obtain another expression of $\Gamma[I_mf,I_ng]=\hat{\mathbb{E}}J_m(f^{\flat_m})J_n(g^{\flat_n})$ to be compared with (\ref{Gamma5}).
}\end{Rq}

\end{document}